\documentclass[10pt,english]{smfart}

\usepackage[T1]{fontenc}
\usepackage[english,francais]{babel}
\usepackage{latexsym,amscd,color}
\usepackage{amsmath,amsfonts,amssymb,mathrsfs}
\usepackage{enumerate,euscript}
\usepackage{amssymb,url,xspace,smfthm}
\input xy
\xyoption{all}

\newcommand{\BibTeX}{{\scshape Bib}\kern-.08em\TeX}
\newcommand{\T}{\S\kern .15em\relax }
\newcommand{\AMS}{$\mathcal{A}$\kern-.1667em\lower.5ex\hbox
        {$\mathcal{M}$}\kern-.125em$\mathcal{S}$}

\DeclareMathOperator{\des}{des}

\DeclareMathOperator{\Hom}{Hom}

\DeclareMathOperator{\rang}{rk}

\DeclareMathOperator{\Spec}{Spec}
\def\sbullet{{\scriptscriptstyle\bullet}}

\title{Distribution of logarithmic spectra of the equilibrium energy}
\date{\today}
\author{Huayi Chen}
\address{Institut Fourier, Universit\'e Joseph Fourier -- Grenoble I}
\email{}
\urladdr{}
\author{Catriona MacLean}
\address{Institut Fourier, Universit\'e Joseph Fourier -- Grenoble I}
\email{}
\urladdr{}

\begin{document}
\def\smfbyname{}
\begin{abstract}
Let $L$ be a big invertible sheaf on a complex projective variety, equipped with two continuous metrics. We prove 
that the distribution of the eigenvalues of the transition matrix between the $L^2$ norms on $H^0(X,nL)$ with respect 
to the two metriques converges (in law) as $n$ goes to infinity to a Borel probability measure on $\mathbb R$. This 
result can be thought of as a generalization of the existence of the energy at the equilibrium as a limit,
or an extension of Berndtsson's results to the more general context of graded linear series and a more general class 
of line bundles. Our approach also 
enables us to obtain a $p$-adic analogue of our main result.
\end{abstract}
\maketitle

\tableofcontents

\section{Introduction}

Let $X$ be a complex projective variety (or alternatively, a projective variety defined over a complete non-archimedean field),
and let $L$ be an invertible sheaf on $X$. We assume that $L$ is \emph{big}, or in other words, that the 
\emph{volume} of $L$, defined by
\[\mathrm{vol}(L):=\lim_{n\rightarrow+\infty}\frac{\rang_{\mathbb C} H^0(X,nL)}{n^d/d!},\]
is strictly positive. (In the above formula, $d$ is the dimension of $X$.) 
In \cite{Ber_Bou08}, Berman and Boucksom studied
the Monge-Amp\`ere energy functional of two continuous metrics on $L$, 
$\varphi$ and $\psi$. The equilibrium Monge-Amp\`ere energy of the pair $(\varphi, \psi)$ is an invariant defined by
\[\mathcal E_{\mathrm{eq}}(\varphi,\psi):=\frac{1}{d+1}\sum_{j=0}^d\int_{X(\mathbb C)}
(P\varphi-P\psi)\,c_1(L,P\psi)^jc_1(L,P\varphi)^{d-j}.\]
In the above formula  $P\phi$ denotes, for any continuous metric $\phi$ on $L$, the supremum of all (possibly singular)
plurisubharmonic continuous metrics bounded above by $\phi$. This invariant describes the asymptotic
behaviour of the ratio of the volumes of the unit balls in the linear systems $H^0(X,nL)$ ($n\geqslant 1$) with respect
to the $L^2$ norms induced by the metrics $\varphi$ and $\psi$, respectively. More precisely, given any probability 
measure $\mu$ on $X(\mathbb C)$, equivalent to Lesbegue measure in every local chart, we have that 
(cf. \cite[Th\'eor\`eme A]{Ber_Bou08})
\begin{equation}\label{Equ:convergenceE}\lim_{n\rightarrow+\infty}\frac{d!}{n^{d+1}}\ln\frac{\mathrm{vol}(\mathcal B^2(nL,\varphi,\mu))}{\mathrm{vol}(\mathcal B^2(nL,\psi,\mu))}=\mathcal E_{\mathrm{eq}}(\varphi,\psi),\end{equation}
where for any continuous metric $\phi$ on $L$, the expression $\mathcal B^2(nL,\phi,\mu)$ denotes the unit ball
in  $H^0(X,nL)$ with respect to the following  $L^2$ norm :
\[\forall\, s\in H^0(X,nL),\quad\|s\|^2_{L^2_{\phi,\mu}}:=\int_{X(\mathbb C)}\|s\|^2_{n\phi}(x)\,\mu(\mathrm{d}x) .\]
All volumes are calculated with respect to some arbitrary Haar measure on $H^0(X,nL)$. 

The logarithm which appears in the left hand side of equation  (\ref{Equ:convergenceE}) is equal (up to
multiplication by a constant) to the mean of the logarithms of the eigenvalues of the transition matrix between
$\varphi_n$ and $\psi_n$. Here $\varphi_n$ and $\psi_n$ are the $L^2$ norms induced by $\varphi$ and $\psi$ respectively
on the space $H^0(X, nL)$.

In \cite{Ber09}, Berndtson studies the distribution of eigenvalues of this transition matrix in the case where
$\varphi$ and $\psi$ are K\"ahler metrics on an ample line bundle $L$. He establishes the folowing fact by
a detailled analysis of geodesics in the space of K\"ahler metrics on $L$ : if $d_n={\rm dim} V_n$ and the numbers
$\lambda_j$ are the logarithms of the eigenvalues of the transition matrics between $\varphi_n$ and $\psi_n$ then the
probability measure \[\nu_n= d_n^{-1}\sum_j \delta_{\lambda_j}\] converges as $n$ tends to infinity to a measure on $\mathbb{R}$ 
which is defined using the Monge-Amp\`ere geodesic linking $\varphi$ and $\psi$ in $\mathcal{H}_L$, the space of
Hermitian metrics in $L$.

The log-ratio of the volumes appearing in \eqref{Equ:convergenceE} is an analogue of the degree function in Arakelov 
geometry. Let $E$ be a vector space of finite rank over a number field $K$ equipped with a family of norms $\|.\|_v$, where $v$ runs 
over the set $M_K$ of places of $K$, and $\|.\|_v$ is a norm on $E\otimes_{K}K_v$ which is ultrametric when $v$ is a 
finite place and Euclidean or Hermitian when $v$ is real or complex. Moreover, we suppose that $E$ contains 
a lattice $\mathcal E$ (ie. a maximal rank sub-$\mathcal O_K$ module, where $\mathcal O_K$ is the ring of algebraic
integers in $K$) such that for all but a finite number of places $v$ the norm $\|.\|_v$ comes from the 
$\mathcal O_K$-module structure on $\mathcal E$.

The data $\overline E=(E,(\|.\|_v))$ is called a \emph{Hermitian adelic bundle} on $K$ and its \emph{Arakelov degree} 
is defined to be the weighted sum
\begin{equation}\label{Equ:degre Arakelov}\widehat{\deg}(\overline E):=-\sum_{v\in M_K}n_v\ln\|s_1\wedge\cdots\wedge s_r\|_v, \end{equation}
where $(s_1,\ldots,s_r)$ is a basis of $E$ over $K$ and the weight $n_v$ is the rank of $K_v$ as a vector space over
$\mathbb Q_v$. The Arakelov degree is well-defined due to the product formula 
\[\forall\,a\in K^{\times},\quad\sum_{v\in M_K}n_v\ln|a|_v=0.\] 

We refer the reader to \cite[Appendice A]{BostBour96} and \cite{Gaudron08} for a detailled exposition of this theory. 
In the context of complex geometry, we consider a finite dimensional vector space $V$ equipped with two Hermitian norms
$\varphi$ and $\psi$. We can write the log-ratio of the volumes of the unit balls (with respect to $\varphi$ and
$\psi$ respectively) as
\begin{equation}\label{Equ:degreArake}\ln\frac{\mathrm{vol}(\mathcal B(V,\varphi))}{\mathrm{vol}(\mathcal B(V,\psi))}=\widehat{\deg}(V,\varphi,\psi):=-\ln\|s_1\wedge\cdots\wedge s_r\|_{\varphi}+\ln\|s_1\wedge\cdots\wedge s_r\|_{\psi},\end{equation}
where $(s_1,\ldots,s_r)$ is a basis for $V$. This expression is independent of the choice of  basis by the elementary 
product formula
\[\forall\,a\in\mathbb C^{\times},\quad \ln|a|-\ln|a|=0.\] 
With this notation, equation \eqref{Equ:convergenceE} can be rewritten iin the form
\[\lim_{n\rightarrow+\infty}\frac{\widehat{\deg}(H^0(X,nL),L^{2}_{\varphi,\mu},L^2_{\psi,\mu})}{n^{d+1}/d!}=\mathcal E_{\mathrm{eq}}(\varphi,
\psi).\]
which is an analogue of  the arithmetic Hilbert-Samuel theorem for Hermitian invertible bundles on a projective 
arithmetic variety.

Given the results in \cite{Ber09} and the convergences proved in \cite{Chen10b}, it is natural to wonder whether the 
eigenvalue distribution of the metric $L_{\varphi,\mu}^2$ with respect to $L_{\psi,\mu}^2$ is well-behaved asymptotically
in more general situations. Unlike the arithmetic case, it does not seem possible to reformulate these spectra 
functorially, which is a key step in the proof of the convergence results in \cite{Chen10b}. This phenomenon arises
essentially because of the presence of a negative weight in formula \eqref{Equ:degreArake}, absent in the
arithmetic case. The interested reader will find counter-examples illustrating this phenomenon in \cite{Chen_hn}, 
remark 5.10 and \S4, example 2.

The main contribution of this article is the introduction of a truncation method for studying the asymptotic behaviour 
of eigenvalues of transition matrices between two metrics. We consider truncations of $\psi$ by dilatations of
$\varphi$, which enables us to prove the following result (cf. theorem \ref{Thm:thmprincipal} \emph{infra})~:

\begin{enonce*}{Main theorem}
Let $k$ be the field of complex numbers or a complete non-archimedean field, and let $X$ be a projective $k$-variety
of dimension $d\geqslant 1$. Let $L$ be a big invertible $\mathcal O_X$-module with two continuous metrics $\varphi$ 
and $\psi$. For any integer $n\geqslant 1$, let $\varphi_n$ and $\psi_n$ be the sup norms on $H^0(X,nL)$ with respect 
to metrics  $\varphi$ and $\psi$ respectively. Moreover, let $\varphi_n'$ and $\psi_n'$ be hermitian norms on 
$H^0(X,nL)$ such that\footnote{If $\eta$ and $\eta'$ are two norms on a finite-dimensional complex vector space $V$, 
then $d(\eta,\eta')$ is defined to be $\displaystyle\sup_{0\neq s\in V}|\ln\|s\|_\eta-\ln\|s\|_{\eta'}|$. This is a 
distance on the set of norms on $V$.} $\max(d(\varphi_n,\varphi_n'),d(\psi_n,\psi_n'))=o(n)$. Let $Z_n$ be the map from
$\{1,\ldots,h^0(X,nL)\}$ to $\mathbb R$ sending $i$ to the  logarithm of the $i^{\text{th}}$ 
eigenvalue (with multiplicity) of $\psi_n'$
with respect to $\varphi_n'$, considered as a random variable on the set $\{1,\ldots,h^0(X,nL)\}$ with the uniform 
distribution. Then the sequence of random variables $(Z_n/n)_{n\geqslant 1}$ converges in law to a probability 
distribution on $\mathbb R$ depending only on the pair $(\varphi,\psi)$.
\end{enonce*}

We recall that by definition the sequence of random variables $Z_n/n$ converges in law if for any continuous bounded 
function $h$ defined on $\mathbb R$ the sequence $(\mathbb E[h(Z_n/n)])_{n\geqslant 1}$ converges in $\mathbb R$.

The theorem \ref{Thm:thmprincipal} proved in \S5 is a little bit stronger than this statement.
It is valid for any sub-graded linear system  of subspaces $V_n\subset H^0(X,nL)$ satisfying conditions (a)-(c) of
section \S4.3. Moreover, it also applies to the functions $\widehat{\mu}_i$ defined in \S2.3, and which are
associated to pairs of (possibly non-hermitian) norms. When these norms are in fact hermitian, these functions are
equal to the logarithms of the eigenvalues of the transition matrix (cf. \S\ref{Sec:relationavecvp}).
This asymptotic distribution is a fine geometric invariant which measures the degree of non-proportionality of the
metrics $\varphi$ and $\psi$. It should be useful in the variational study of metrics on an invertible sheaf.

The proof of the main theorem uses various techniques drawn from algebraic and arithmetic geometry. In\S2-3 we 
introduce a Harder-Narisimhan type theory for finite dimensional vector spaces with two norms. This can be thought of 
as a geometric reformulation of the eigenvalues of the transition matrix between two Hermitian norms. This construction
has the advantage of being valid for non-Hermitian norms, which allows us to work directly with the sup norm. Moreover,
it is analogous to the Harder-Narasimhan filtration and polygon of a vector bundle on a smooth projective curve. 
In particular, the truncation results  (propositions \ref{Pro:troncatureC} and \ref{Pro:troncaturepadique}) are crucial
for the proof of the main theorem. Another important ingredient is the existence of the equilibrium energy as a limit,
presented in \S\ref{Sec:energiealequilibre}. For a complete complex graded system, this is a result of Berman and Boucksom 
\cite{Ber_Bou08} (where the limit is described.) Here, we use the Okounkov filtered semi-groups point of view 
developped in \cite{Boucksom_Chen}, analogous to that of Witt-Nystrom \cite{Nystrom11}. The combination of this method
with the Harder-Naransimhan formalism is extremely flexible and enables us to prove the existence of the equilibrium 
energy as a limit in the very general setting of a graded linear system on both complex and non-archimedean varieties
(cf. theorem \ref{Thm:energieequilibree} and its corollary \ref{Cor:thmlimite}). Finally, in Section 
\S\ref{Sec:demduthm} we prove a general version of the main theorem (cf. theorem \ref{Thm:thmprincipal} and remark
\ref{Rem:modification de metrique}).
 
\section{Slopes of a vector space equipped with two norms.}

In this section we develop the formalism of slopes and Harder-Narasimham filtrations for finite dimensional complex vector 
spaces equipped with a pair of  norms.
\subsection{Slopes and the Harder-Narasimhan filtration.}
Let $\mathcal C^H$ be the class of triplets $\overline{V}=(V,\varphi,\psi)$, where $V$ is a finite dimensional 
complex vector space and $\varphi$ and $\psi$ are two Hermitian norms on $V$. For any 
$\overline V=(V,\varphi,\psi)\in\mathcal C^H$, we let $\widehat{\deg}(\overline V)$ be the real number defined by
\begin{equation}\label{Equ:degree}-\ln\|s_1\wedge\cdots\wedge s_r\|_\varphi+\ln\|s_1\wedge\cdots\wedge s_r\|_{\psi},\end{equation}
where $(s_1,\ldots,s_r)$ is a basis of $V$. If the $V$ is non trivial then we let $\widehat{\mu}(\overline V)$ be the
quotient $\widehat{\deg}(\overline V)/\rang(V)$, which we call the \emph{slope} of $\overline V$. Unless otherwise 
specified, for any subspace $W$ of $V$ we will let $\overline W$ be the vector space  $W$ equipped with the induced
norms and  $\overline V/\overline W$ be the quotient space $V/W$ equipped with the quotient norms. The following 
relationship holds for any subspace $W$ of $V$~:
\begin{equation}\label{Equ:suite exacte}
\widehat{\deg}(\overline V)=\widehat{\deg}(\overline W)+\widehat{\deg}(\overline V/\overline W).
\end{equation}
\begin{prop}
Let $\overline V=(V,\varphi,\psi)$ be a non-trivial element of  $\mathcal C^H$. There is a unique subspace $V_{\des}$ in
$V$ satisfying the following properties:
\begin{enumerate}[(1)]
\item for any subspace $W\subset V$ we have that 
$\widehat{\mu}(\overline W)\leqslant\widehat{\mu}(\overline{V}_{\!\des})$,
\item if $W$ is a subspace of $V$ such that $\widehat{\mu}(\overline W)=\widehat{\mu}(\overline{V}_{\!\des})$ then
$W\subset V_{\des}$. 
\end{enumerate}
\end{prop}

\begin{proof}
Let $\lambda$ be the norm of the identity map from $(V,\|.\|_\varphi)$ to $(V,\|.\|_\psi)$. We will prove that the set 
\[V_{\des}=\{x\in V\,:\,\|x\|_{\psi}=\lambda \|x\|_{\varphi}\}\]
satisfies the conditions of the proposition. We start by checking that  $V_{\des}$ is a non-trivial subspace of $V$. By definition,  $V_{\des}$ contains at least one non-zero vector and is stable under multiplication by a complex scalar. We
need to check that  $V_{\des}$ is stable under addition. Let $x$ and $y$ be two elements of $V_{\des}$. As the norms 
$\varphi$ and $\psi$ are Hermitian, the parallelogram law states that 
\begin{gather*}
\|x+y\|_{\varphi}+\|x-y\|_{\varphi}=2(\|x\|_{\varphi}+\|y\|_{
\varphi}),\\
\|x+y\|_{\psi}+\|x-y\|_{\psi}=2(\|x\|_{\psi}+\|y\|_{\psi}).
\end{gather*}

Since $\lambda$ is the norm of the identity map from $(V,\|.\|_\varphi)$ to $(V,\|.\|_\psi)$ we have that 
$\|x+y\|_{\psi}\leqslant \lambda\|x+y\|_{\varphi}$ and $\|x-y\|_{\psi}\leqslant \lambda\|x-y\|_{\varphi}$. As $x$ and $y$ are
vectors in  $V_{\des}$ we have that $\|x\|_{\psi}=\lambda\|x\|_\varphi$ and $\|y\|_{\psi}=\lambda\|y\|_\varphi$. It follows
that
\[\begin{split}&\quad\;2\lambda(\|x\|_{\varphi}+\|y\|_{\varphi})=2(\|x\|_{\psi}+\|y\|_{
\psi})=\|x+y\|_{\psi}+\|x-y\|_{\psi}\\
&\leqslant\lambda(\|x+y\|_{\varphi}+\|x-y\|_{\varphi})=2\lambda(\|x\|_{\varphi}+\|y\|_{\varphi}),
\end{split}\]
so $\|x+y\|_{\psi}=\lambda\|x+y\|_{\varphi}$ and it follows that $x+y\in V_{\des}$.

In particular, as the norms $\|.\|_\psi$ and $\|.\|_\varphi$ are proportional with ratio $\lambda$ on  $V_{\des}$ we 
have that $\widehat{\mu}(\overline V_{\!\des})=\ln(\lambda)$. If $W$ is a subspace of $V$ then the identity map 
from $(\Lambda^rW,\|.\|_{\psi})$ to $(\Lambda^rW,\|.\|_{\varphi})$ has norm $\leqslant\lambda^r$ (by Hadamard's 
inequality). This inequality is an equality if and only if  the norms $\|.\|_\psi$ and $\|.\|_{\varphi}$ are proportional
with ratio $\lambda$ - or in other words, the space $W$ is contained in $V_{\des}$.
\end{proof}

Let $\overline V$ be a non-trivial element of $\mathcal C^H$. Let $\widehat{\mu}_{\max}(\overline V)$ be the slope of 
$\overline V_{\!\des}$, which we call the \emph{maximal slope} of $\overline V$. We say that $\overline V$ is
\emph{semi-stable} if $\widehat{\mu}_{\max}(\overline V)=\widehat{\mu}(\overline V)$, or equivalently $V_{\des}=V$. The
element $\overline V$ is semi-stable if and only if the two norms on $\overline V$ are proportional.

For any  $\overline V$, non-trivial element of $\mathcal C^H$, we  construct recursively a sequence of
subspaces of $V$ of the form
\begin{equation}\label{Equ:HN}0=V_0\subset V_1\subset\ldots\subset V_n=V\end{equation}
such that $V_{i}/V_{i-1}=(V/V_{i-1})_{\des}$ for any $i\in\{1,\ldots,n\}$ using the quotient norms on $V/V_{i-1}$. Each of 
the subquotients $\overline V_{\!i}/\overline V_{\!i-1}$ is a semi-stable element of $\mathcal C^H$. Moreover, if 
$\mu_i$ is the slope of $\overline V_{\!i}/\overline V_{\!i-1}$, then we have that:
\begin{equation}\label{Equ:pentes}
\mu_1>\mu_2>\ldots>\mu_n.
\end{equation}
These numbers are called the \emph{intermediate slopes} of $\overline V$. The flag \eqref{Equ:HN} is called the 
\emph{Harder-Narasimhan filtration} of $\overline V$.

\begin{defi}\label{Def:va}
Let $\overline V$ be a non-zero element of $\mathcal C^H$  with its Harder-Narasimham filtration and 
intermediate slopes as defined in \eqref{Equ:HN} and \eqref{Equ:pentes}. We let $Z_{\overline V}$ be the
random variable with values in $\{\mu_1,\ldots,\mu_n\}$ such that 
\[\mathbb P(Z_{\overline V}=\mu_i)=\frac{\rang(V_i/V_{i-1})}{\rang(V)}\] for any $i\in\{1,\ldots,n\}$.
\end{defi}

\begin{rema}
The above constructions are analogues of Harder-Narasimhan theory for vector bundles on a regular projective curve or
hermitian adelic bundles on a number field. As in  \cite[\S2.2.2]{Chen10b}, we can include the Harder-Narasimhan 
filtration and the intermediate slopes in a decreasing $\mathbb R$ filtration. However, contrary to the geometric and
arithmetic cases, this $\mathbb R$-filtration is not functorial, as explained in \cite[remarque 5.10]{Chen_hn}.
Moreover, the Harder-Narasimhan filtration is not necessarily the only filtration whose sub-quotients are semi-stable
with strictly decreasing slopes, as can be seen using the counter-example in \cite[\S4 exemple 2]{Chen_hn}. 
\end{rema}

Let $\overline V$ be an element of $\mathcal C^H$ of rank $r>0$. We let $\widetilde P_{\overline V}$ be the function on
$[0,r]$ whose graph is the upper boundary of the convex closure of the set of points of the form 
$(\rang(W),\widehat{\deg}(\overline W))$. This is a concave function which is affine on each interval $[i-1,i]$ 
($i\in\{1,\ldots,r\}$). We call it the Harder-Narasimhan \emph{polygon} of $\overline V$. By definition, 
$\widetilde P_{\overline V}(0)=0$ and $\widetilde P_{\overline V}(r)=\widehat{\deg}(V)$. For any $i\in\{1,\ldots,r\}$, 
we let $\widehat{\mu}_i(\overline V)$ be the slope of the function $\widetilde P_{\overline V}$ on the interval 
$[i-1, i]$, which we call the  $i^{\text{th}}$ \emph{slope} of $\overline V$. We also introduce a normalised version of 
the Harder-Narasimhan polygon, defined by
\[P_{\overline V}(t):=\frac{1}{\rang(V)}\widetilde P_{\overline V}(t\rang(V)),\quad t\in[0,1].\] 
It follows from the relation $\widetilde P_{\overline{V}}(r)=\widehat{\deg}(V)$ that 
\begin{equation}\label{Equ:somme des pentes}\widehat{\deg}(\overline V)=\widehat{\mu}_1(\overline V)+\cdots+\widehat{\mu}_r(\overline V).\end{equation}
Moreover, the distribution of the random variable $Z_{\overline V}$ (cf. definition \ref{Def:va}) is given by 
\[\frac 1r\sum_{i=1}^r\delta_{\widehat{\mu}_i(\overline V)},\]
where $\delta_a$ is a Dirac measure supported at $a$.

The normalised polygon $P_{\overline V}$, the random variable $Z_{\overline V}$ and the slopes of $\overline V$ 
are linked by the following simple formula.
\begin{equation}\label{Equ:egalite fondamentale}
P_{\overline V}(1)=\mathbb E[Z_{\overline V}]=\widehat{\mu}(\overline V).
\end{equation}
\subsection{The link with eigenvalues.}\label{Sec:relationavecvp}
The Harder-Narasimhan polygon and its intermediate slopes can be thought of as an intrinsic interpretation of the 
eigenvalues and eigenspaces of the transition matrix between two Hermitian norms. Indeed, given an element 
$\overline{V}=(V,\varphi,\psi)\in\mathcal C^H$ and an orthonormal basis $\boldsymbol{e}=(e_i)_{i=1}^r$ for 
$V$ with respect to $\|.\|_{\varphi}$ we can construct a Hermitian matrix   
\[A_{\overline V\!,\boldsymbol{e}}=\begin{pmatrix}
\langle e_1,e_1\rangle_{\psi}&\langle e_1,e_2\rangle_{\psi}&\cdots&\langle e_1,e_r\rangle_{\psi}\\
\langle e_2,e_1\rangle_{\psi}&\langle e_2,e_2\rangle_{\psi}&\cdots&\langle e_1,e_r\rangle_{\psi}\\
\vdots&\vdots&\ddots&\vdots\\
\langle e_r,e_1\rangle_{\psi}&\langle e_r,e_2\rangle_{\psi}&\cdots&\langle e_r,e_r\rangle_{\psi}
\end{pmatrix}\]
If the intermediate slopes of $\overline V$ are  $\mu_1>\mu_2>\ldots>\mu_n$ then the eigenvalues of 
$A_{\overline V\!,\boldsymbol{e}}$ are $\mathrm{e}^{2\mu_1},\ldots,\mathrm{e}^{2\mu_n}$. Moreover, if for any 
$i\in\{1,\ldots,n\}$ we let  $V^{(i)}$ be the eigenspace associated to the eigenvalue $\mathrm{e}^{2\mu_i}$ of the 
endomorphism of $V$ given by the matrix $A_{\overline V\!,\boldsymbol{e}}$ in the basis $\boldsymbol{e}$ then the flag
\[0\subsetneq V^{(1)}\subsetneq V^{(1)}+V^{(2)}\subsetneq\ldots
\subsetneq V^{(1)}+\cdots+V^{(n)}=V\]
is the Harder-Narasimhan filtration of $\overline V$.

The fact that $\varphi$ and $\psi$ are proportional on each of the subspaces $V^{(i)}$ implies that there is a complete
flag (which will be in general finer than the Harder-Narasimhan filtration of $\overline V$)
\[0=V_0\subsetneq V_1\subsetneq \ldots \subsetneq V_r=V\]
such that $\widehat{\deg}(\overline V_i/\overline V_{\!i-1})=\widehat{\mu}_i(\overline V)$. In particular, $P_{\overline V}(i)=\widehat{\deg}(\overline V_{\!i})$. This can be thought of as a version of the Courant-Fischer theorem for 
symmetric or Hermitian matrices. 
\subsection{Generalisation to arbitrary norms}\label{Subsec:polygonenorme} The above constructions can be naturally
generalised to spaces equipped with two arbitrary norms. Let $\mathcal C$ be the class of finite-dimensional complex 
vector spaces equipped with two (not necessarily Hermitian) norms. For any non-trivial element 
$\overline V=(V,\varphi,\psi)$ in $\mathcal C$ we let $\widehat{\deg}(\overline V)$ be the number 
\[\ln\frac{\mathrm{vol}(\mathcal B(V,\varphi))}{\mathrm{vol}(\mathcal B(V,\psi))},\]
where $\mathcal B(V,\varphi)$ and $\mathcal B(V,\psi)$ are the unit balls with respect to the norms $\varphi$ and 
$\psi$ respectively, and $\mathrm{vol}$ is a Haar measure on $V$. This definition is independent of the choice of Haar 
measure $\mathrm{vol}$. Moreover, when the norms $\varphi$ and $\psi$ are Hermitian, it is equal to the number defined 
in \eqref{Equ:degree}. As in the Hermitian case, we let $\widetilde P_{\overline V}$ be the concave function on 
$[0,\rang(V)]$ whose graph is the upper boundary of the convex closure of the set of points of the form 
$(\rang(W),\widehat{\deg}(W))$, where $W$ is a member of the set of subspaces of $V$. For any 
$i\in\{1,\ldots,\rang(V)\}$, we let $\widehat{\mu}_i(\overline V)$ be the slope of the function 
$\widetilde P_{\overline V}$ on the interval $[i-1,i]$. The equality \eqref{Equ:somme des pentes} holds in this more general
context. We let $Z_{\overline V}$ be a random variable whose law is an average of Dirac masses at the intermediate
slopes $\widehat{\mu}_i(\overline V)$~:
\[\text{the law of $Z_{\overline V}$ is }\frac{1}{\rang(V)}\sum_{i=1}^{\rang(V)}\delta_{\widehat{\mu}_i(\overline V)}.\]
We can also introduce a normalised Harder-Narasimhan polygon :
\begin{equation}\label{Equ:polygone normalise}P_{\overline V}(t)=\frac{1}{\rang(V)}\widetilde P_{\overline V}(t\rang(V)),\quad t\in[0,1].\end{equation}
Equation \eqref{Equ:egalite fondamentale} holds in this more general setting : we have that 
\begin{equation} \label{Equ:egalite fondamentale2}
P_{\overline V}(1)=\mathbb E[Z_{\overline V}]=\widehat{\mu}(\overline V).
\end{equation}
The following result compares Harder-Narasimhan polygons.
\begin{prop}\label{Pro:comparaisonnorme}
Let $(V,\varphi,\psi)$ be an element of $\mathcal C$. If $\psi'$ is another norm on $V$ such that\footnote{In other 
words, for any $x\in V$ we have that $\|x\|_{\psi'}\leqslant\|x\|_{\psi}$.} $\psi'\leqslant\psi$ then we have that 
$\widetilde P_{(V,\varphi,\psi')}\leqslant \widetilde P_{(V,\varphi,\psi)}$ as functions on $[0,\rang(V)]$.
\end{prop}
\begin{proof}
For any subspace $W$ on $V$ we have that $\mathcal B(W,\psi)\subset\mathcal B(W,\psi')$ and it follows that 
$\widehat{\deg}(W,\varphi,\psi)\geq\widehat{\deg}(W,\varphi,\psi')$. The point $(\rang(W),\widehat{\deg}(W,\varphi,\psi'))$ is therefore always below the graph of $\widetilde P_{(V,\varphi,\psi)}$. It follows that 
$\widetilde P_{(V,\varphi,\psi')}\leqslant \widetilde P_{(V,\varphi,\psi)}$.
\end{proof}
\begin{rema}
Similarly, if $(V,\varphi,\psi)$ is an element of $\mathcal C$ and $\varphi'$ is another norm on $V$ such that 
$\varphi'\geqslant\varphi$ then we have that $\widetilde P_{(V,\varphi',\psi)}\leqslant \widetilde P_{(V,\varphi,\psi)}$ as
functions on $[0,\rang(V)]$.
\end{rema}
If $V$ is a non-trivial finite dimensional complex vector space and $\psi$ and $\psi'$ are two norms on $V$ then we
denote by $d(\psi,\psi')$ the quantity
\[\sup_{0\neq x\in V}\big|\ln\|x\|_{\psi}-\ln\|x\|_{\psi'}\big|.\]
\begin{coro}\label{Cor:compraisonpolygone}
Let $(V,\varphi,\psi)$ be a non-trivial element of $\mathcal C$ and let $\varphi'$ and $\psi'$ be two norms on $V$. For
any $t\in [0,\rang(V)]$  we have that 
\[|\widetilde P_{(V,\varphi,\psi)}(t)-\widetilde P_{(V,\varphi',\psi')}(t)|\leqslant (d(\varphi,\varphi')+d(\psi,\psi'))\,t.\]
\end{coro}
\begin{proof}
We denote by $\psi_1$ and $\psi_2$ the norms on $V$ such that
\[\|.\|_{\psi_1}=\mathrm{e}^{-d(\psi,\psi')}\|.\|_{\psi}\text{ and }
\|.\|_{\psi_2}=\mathrm{e}^{d(\psi,\psi')}\|.\|_{\psi}.\]
We have that $\psi_1\leqslant\psi'\leqslant\psi_2$. Moreover, for any $t\in[0,1]$ we have that
\[\widetilde P_{(V,\varphi,\psi_1)}(t)=\widetilde P_{(V,\varphi,\psi)}(t)-d(\psi,\psi')\,t,\quad \widetilde P_{(V,\varphi,\psi_1)}(t)=\widetilde P_{(V,\varphi,\psi)}(t)+d(\psi,\psi')\,t.\]
By the above proposition we have that 
\[|\widetilde P_{(V,\varphi,\psi)}(t)-\widetilde P_{(V,\varphi,\psi')}(t)|\leqslant d(\psi,\psi')\,t.\]
By the same argument we have that 
\[|\widetilde P_{(V,\varphi,\psi')}(t)-\widetilde P_{(V,\varphi',\psi')}(t)|\leqslant d(\varphi,\varphi')\,t.\]
The sum of these two inequalities gives the required result.
\end{proof}
\subsection{John and L\"owner norms} Whilst the elements of $\mathcal C$ do not generally satisfy 
\eqref{Equ:suite exacte}, John and L\"owner's ellipsoid technique enables us to show that \eqref{Equ:suite exacte}  holds up
to an error term. Indeed, if $V$ is a complex vector space of rank $r>0$ and $\phi$ is a norm on $V$ then we can find 
two Hermitian norms $\phi_L$ and $\phi_J$ which satisfy the following properties (cf. \cite[page 84]{Thompson96})~:
\[\frac{1}{\sqrt{r}}\|.\|_{\phi_J}\leqslant\|.\|_{\phi}
\leqslant\|.\|_{\phi_J},\quad
\|.\|_{\phi_L}\leqslant\|.\|_{\phi}\leqslant\sqrt{r}\|.\|_{\phi_L}\]
\begin{prop}\label{Pro:comparaisondepolygone}
Let $\overline V=(V,\varphi,\psi)$ be an element of $\mathcal C$. If 
\[0=V_0\subsetneq V_1\subsetneq\ldots\subsetneq V_n=V\]
is a flag of subspaces of $V$ then we have that
\begin{equation}\label{Equ:sommesousquotient}\bigg|\widehat{\deg}(V)-\sum_{i=1}^n\widehat{\deg}(\overline V_i/
\overline V_{i-1})\bigg|\leqslant\rang(V)\ln(\rang(V)).\end{equation}
\end{prop}
\begin{proof}
Consider the object $(V,\varphi_L,\psi_J)$, which is an element of $\mathcal C^H$. The following equation holds by formula 
\eqref{Equ:suite exacte}): 
\[\widehat{\deg}(V,\varphi_L,\psi_J)=\sum_{i=1}^n\widehat{\deg}(V_i/V_{i-1},\varphi_L,\psi_J).\]
Moreover, since $\varphi_L\leqslant\varphi$ and $\psi_J\geqslant\psi$, we have that
\[\widehat{\deg}(V_i/V_{i-1},\varphi,\psi)\leqslant\widehat{\deg}(V_i/V_{i-1},\varphi_L,\psi_J).\]
It follows that
\begin{equation}\label{Equ:degfil}\widehat{\deg}(V,\varphi_L,\psi_J)\geqslant
\sum_{i=1}^n\widehat{\deg}(V_i/V_{i-1},\varphi,\psi).\end{equation}
Moreover, it follows from the relations \[\|.\|_{\varphi}\leqslant\sqrt{\rang(V)}\|.\|_{\varphi_L}\quad\text{and}\quad
 \|.\|_{\psi}\geqslant (\sqrt{\rang(V)})^{-1}\|.\|_{\psi_J}\] that
\begin{equation}\label{Equ:degfil2}\widehat{\deg}(V,\varphi,\psi)\geqslant\widehat{\deg}(V,\varphi_L,\psi_J)-\rang(V)\ln(\rang(V)).\end{equation}
It follows from the inequalities \eqref{Equ:degfil} and \eqref{Equ:degfil2} that
\[\widehat{\deg}(\overline V)\geqslant\sum_{i=1}^n\widehat{\deg}(\overline V_i/\overline V_{i-1})-\rang(V)\ln(\rang(V)).\]
Applying the same argument to $(V,\varphi_J,\psi_L)$ we can show that
\[\widehat{\deg}(\overline V)\leqslant\sum_{i=1}^n\widehat{\deg}(\overline V_i/\overline V_{i-1})+\rang(V)\ln(\rang(V)).\]
This completes the proof of the proposition.
\end{proof}
\subsection{Truncation}\label{SubSec:Troncature} Let $V$ be a finite-dimensional complex vector space. If $\varphi$ 
is a norm on $V$ and $a$ is a real number then we let  $\varphi(a)$ be the norm on $V$ such that
\[\forall\,x\in V,\quad \|x\|_{\varphi(a)}=\mathrm{e}^{a}\|x\|_{\varphi}.\]
If $\varphi$ and $\psi$ are two norms on $V$ then we denote by $\varphi\vee\psi$ the norm on $V$ such that
\[\forall\,x\in V,\quad \|x\|_{\varphi\vee\psi}=\max(\|x\|_{\varphi},\|x\|_{\psi}).\]
\begin{prop}\label{Pro:troncatureC}
Let $\overline V=(V,\varphi,\psi)$ be a non-zero element of $\mathcal C$ and let $a$ be a real number. We have that 
\[\bigg|\widehat{\deg}(V,\varphi,\psi\vee\varphi(a))-\sum_{i=1}^{\rang(V)}\max(\widehat{\mu}_i(\overline V),a)\bigg|\leqslant 2\rang(V)\ln(\rang(V))+\frac{\rang(V)}{2}\ln(2).\]
\end{prop}
\begin{proof} Let $r$ be the dimension of $V$.
We will first prove the following equation.
\begin{equation}\label{Equ:sommepartiel}\sum_{i=1}^{r}\max(\widehat{\mu}_i(\overline V),a)=\sup_{t\in[0,r]}\Big(\widetilde P_{\overline V}(t)-at\Big)+ar.\end{equation}
As the function $\widetilde P_{\overline V}(t)-at$ is affine on each segment $[i-1,i]$ ($i\in\{1,\ldots,r\}$) we get 
that 
\[\sup_{t\in[0,r]}\big(\widetilde P_{\overline V}(t)-at\big)=\max_{i\in\{0,\ldots,r\}}\bigg(\sum_{1\leqslant j\leqslant i}\big(\widehat{\mu}_i(\overline V)-a\big)\bigg)=\sum_{i=1}^r\max\big(\widehat{\mu}_i(\overline V)-a,0\big).\]
It follows that 
\[\sup_{t\in[0,r]}\big(\widetilde P_{\overline V}(t)-at\big)+ar=\sum_{i=1}^r\Big(\max\big(\widehat{\mu}_i(\overline V)-a,0\big)+a\Big)=\sum_{i=1}^r\max(\widehat{\mu}_i(\overline V),a).\]
We start by proving the proposition in the special case where $\varphi$ and $\psi$ are Hermitian. There is then a basis
$\boldsymbol{e}=(e_i)_{i=1}^r$ which is orthonormal for $\varphi$ and orthogonal for $\psi$. For any 
$i\in\{1,\ldots,r\}$ set $\lambda_i=\|e_i\|_\psi$. Without loss of generality, we may assume that 
$\lambda_1\geqslant\ldots\geqslant\lambda_r$. We therefore have that 
$\widehat{\mu}_i(\overline V)=\ln(\lambda_i)$ for any 
$i\in\{1,\ldots,r\}$. Let $\psi'$ be the Hermitian norm on $V$ such that
\[\|x_1e_1+\cdots+x_re_r\|^2_{\psi'}=\sum_{i=1}^r x_i^2\max(\lambda_i^2,\mathrm{e}^{2a}).\]
We then have that
\[\|x_1e_1+\cdots+x_re_r\|_{\psi\vee\varphi(a)}^2=\max\bigg(\sum_{i=1}^rx_i^2\lambda_i^2,\sum_{i=1}^rx_i^2\mathrm{e}^{2a}\bigg).\]
and it follows that 
\[\|.\|_{\psi\vee\varphi(a)}\leqslant\|.\|_{\psi'}\leqslant\sqrt{2}\|.\|_{\psi\vee\varphi(a)}.\]
It follows that
\[\widehat{\deg}(V,\varphi,\psi\vee\varphi(a))
\leqslant\widehat{\deg}(V,\varphi,\psi')
\leqslant \frac r2\log(2)+\widehat{\deg}(V,\varphi,\psi\vee\varphi(a)),\]
and hence
\[\bigg|\widehat{\deg}(V,\varphi,\psi\vee\varphi(a))-\sum_{i=1}^r\max(\widehat{\mu}_i(\overline V),a)\bigg|\leqslant \frac r2\log(2).\]
We now deal with the general case. We choose Hermitian norms $\varphi_1$ and $\psi_1$ such that 
$d(\varphi,\varphi_1)\leqslant\frac 12\log(r)$ and $d(\psi,\psi_1)\leqslant\frac 12\log(r)$. Applying the above result
to $(V,\varphi_1,\psi_1)$ we get that 
\[\bigg|\deg(V,\varphi_1,\psi_1\vee\varphi_1(a))-\sum_{i=1}^r\max(\widehat{\mu}_i(V,\varphi_1,\psi_1),a)\bigg|\leqslant\frac r2\log(2).\]
Moreover, we have that
\[d(\psi\vee\varphi(a),\psi_1\vee\varphi_1(a))\leqslant\max\big(d(\varphi,\varphi_1),d(\psi,\psi_1)\big)\]
and hence
\[|\deg(V,\varphi,\psi\vee\varphi(a))-\deg(V,\varphi_1,\psi_1\vee\varphi_1(a))|\leqslant d(\varphi,\varphi_1)r+\max(d(\varphi,\varphi_1),d(\psi,\psi_1))r,\]
which is bounded above by $r\ln(r)$.
Moreover, by \ref{Pro:comparaisondepolygone} we get that
\[\big|\widetilde P_{\overline V}(t)-\widetilde P_{(V,\varphi_1,\psi_1)}(t)\big|
\leqslant \big(d(\varphi,\varphi_1)+d(\psi,\psi_1)\big)t\leqslant t\ln(r)\leqslant r\ln(r),\]
and hence
\[\bigg|\sup_{t\in[0,r]}\Big(\widetilde P_{\overline V}(t)-at\Big)-
\sup_{t\in[0,r]}\Big(\widetilde P_{(V,\varphi_1,\psi_1)}(t)-at\Big)\bigg|\leqslant r\ln(r).\]
By \eqref{Equ:sommepartiel} we get that
\[\bigg|\widehat{\deg}(V,\varphi,\psi\vee\varphi(a))-\sum_{i=1}^r\max(\widehat{\mu}_i(\overline V),a)\bigg|\leqslant 2r\ln(r)+\frac r2\log(2).\]
\end{proof}

\section{The non-Archimedean analogue.}
In this section, we develop an analogue for slopes for finite dimensional vector spaces over a non-archimedean field
$k$ equipped with two ultrametric norms. Let $k$ be a field equipped with a complete non-archimedean absolute value 
function $|.|$.

\subsection{Ultrametric norms.}
Let $V$ be a $k$-vector space of dimension $r$ equipped with a norm $\|.\|$. As $k$ is assumed to be complete the 
topology on  $V$ is induced by any isomorphism $k^r\rightarrow V$. (We refer the reader to 
\cite[I.\S2, $\text{n}^\circ$3]{Bourbaki81} theorem 2 and the remark on page I.15 for a proof of this fact). 
In particular, any subspace of $V$ is closed (cf. {\it loc. cit.} 
corollary 1 of theorem 2).

Let $(V,\|.\|)$ be a finite-dimensional ultra-normed vector space on $k$. A basis $\boldsymbol{e}=(e_i)_{i=1}^r$ of $V$
is said to be \emph{orthogonal} if the following holds :
\[\forall\,(\lambda_1,\ldots,\lambda_r)\in k^r,\quad \|\lambda_1e_1+\cdots+\lambda_re_r\|=\max_{i\in\{1,\ldots,r\}}|\lambda_i|\cdot \|e_i\|.\]  
An ultra-normed vector space does not necessarily have an orthogonal basis, but the following proposition shows that an
asymptotic version of Gram-Schmidt''s algorithm is still valid in this context. 

\begin{prop}\label{Pro:existenceepsorth}
Let $(V,\|.\|)$ be a ultra-normed $k$-vector space of dimension $r\geqslant 1$. Let
\[0=V_0\subsetneq V_1\subsetneq V_2\subsetneq\ldots\subsetneq V_r=V\]
be a full flag of subspaces of $V$.
For any $\varepsilon\in\, ]0,1[\,$ there is a basis $\boldsymbol{e}=(e_i)_{i=1}^r$ compatible with the flag\footnote{We 
say that a basis $\boldsymbol{e}$ is compatible with a full flag 
$0=V_0\subsetneq V_1\subsetneq V_2\subsetneq\ldots\subsetneq 
V_r=V$ if for any $i\in\{1,\ldots,r\}$, we have that $\mathrm{card}(V_i\cap\boldsymbol{e})=i$.} such that
\begin{equation}\label{Equ:epsothogonale}\forall\,(\lambda_1,\ldots,\lambda_r)\in k^r,\quad \|\lambda_1e_1+\cdots+\lambda_re_r\|\geqslant (1-\varepsilon)\max_{i\in\{1,\ldots,r\}}|\lambda_i|\cdot \|e_i\|.\end{equation}
\end{prop}

\begin{proof}
We proceed by induction on $r$, the dimension of $V$. The case $r=1$ is trivial. Assume the proposition holds for all
spaces of dimension $<r$ for some $r\geqslant 2$. Applying the induction hypothesis to $V_{r-1}$ and the flag $0=V_0
\subsetneq V_1\subsetneq\ldots\subsetneq V_{r-1}$ we get a basis $(e_1,\ldots,e_{r-1})$ compatible with the flag such 
that
\begin{equation}\label{Equ:hyprecurence}\forall\,(\lambda_1,\ldots,\lambda_{r-1})\in k^{r-1},\quad\|\lambda_1e_1+\cdots+\lambda_{r-1}e_{r-1}\|\geqslant (1-\varepsilon)\max_{i\in\{1,\ldots,r-1\}}|\lambda_i|\cdot\|e_i\|.\end{equation}
Let $x$ be an element of $V\setminus V_{r-1}$ and let $y$ be a point in $V_{r-1}$ such that
\begin{equation}\label{Equ:distance}
\|x-y\|\leqslant (1-\varepsilon)^{-1}\mathrm{dist}(x,V_{r-1}).
\end{equation}
(The distance between $x$ 
and $V_{r-1}$ is strictly positive because $V_{r-1}$ is closed in $V$.)
We choose $e_r=x-y$. The basis $e_1,\ldots,e_r$ is compatible with the flag $0=V_0\subsetneq V_1\subsetneq \ldots
\subsetneq V_r=V$. Let $(\lambda_1,\ldots,\lambda_r)$ be an element of $k^r$: we wish to find a lower bound for the 
norm of $z=\lambda_1e_1+\cdots+\lambda_re_r$. By \eqref{Equ:distance} we have that 
\[\|z\|\geqslant |\lambda_r|\cdot\mathrm{dist}(x,V_{r-1})\geqslant (1-\varepsilon)|\lambda_r|\cdot\|e_r\|.\] 
This provides our lower bound when $\|\lambda_r e_r\|\geqslant \|\lambda_1e_1+\cdots+\lambda_{r-1}e_{r-1}\|$. If 
$\|\lambda_r e_r\|< \|\lambda_1e_1+\cdots+\lambda_{r-1}e_{r-1}\|$ then we have 
$\|z\|=\|\lambda_1e_1+\cdots+\lambda_{r-1}e_{r-1}\|$ because the norm is ultrametric. By the induction hypothesis 
\eqref{Equ:hyprecurence} we have that $\|z\|\geqslant(1-\varepsilon)|\lambda_i|\cdot\|e_i\|$ for any 
$i\in\{1,\ldots,r-1\}$. This completes the proof of the proposition.
\end{proof}

\begin{rema}
If $k$ is locally compact then the above equation holds for $\varepsilon=0$ since the distance appearing in 
\eqref{Equ:distance} is then attained.
\end{rema}

\begin{defi}
Let $V$ be an ultra-normed $k$-vector space and $\alpha$ a real number in $]0,1]$. We say that a basis 
$\boldsymbol{e}=(e_1,\ldots,e_r)$ of $V$ is $\alpha$-\emph{orthogonal} if for any $(\lambda_1,\ldots,\lambda_r)\in k^r$ we have that
\[\|\lambda_1e_1+\cdots+\lambda_re_r\|\geqslant \alpha\max(|\lambda_1|\cdot\|e_1\|,\ldots,|\lambda_r|\cdot\|e_r\|).\]
By definition, $1$-orthogonality is the same thing as orthogonality.
\end{defi}
For any ultra-normed finite dimensional $k$-vector space $V$ we let $V^\vee=\Hom_k(V,k)$ be its dual space with the
operator norm. This is also a finite-dimensional ultra-normed $k$-vector space.
\begin{prop}\label{Pro:alphaorthogonale}
Let $V$ be a finite-dimensional ultranormed $k$-vector space and let $(e_1,\ldots,e_r)$ be an $\alpha$-orthogonal basis
for $V$ for some $\alpha\in\,]0,1]$. The dual basis $(e_1^\vee,\ldots,e_r^\vee)$ is then $\alpha$-orthogonal.
\end{prop}
\begin{proof}
Consider an $r$-tuple $(\lambda_1,\ldots,\lambda_r)\in k^r$. We then have that
\[e_i^{\vee}(\lambda_1e_1+\cdots+\lambda_re_r)=\lambda_i.\]
It follows that
\begin{equation}\label{Equ:normeduale}\|e_i^\vee\|=\sup_{(\lambda_1,\ldots,\lambda_r)\neq(0,\ldots,0)}\frac{|\lambda_i|}{\|\lambda_1e_1+\cdots+\lambda_re_r\|}\leqslant \alpha^{-1}\|e_i\|^{-1}\end{equation}
Consider $\xi=\lambda_1e_1^\vee+\cdots+\lambda_re_r^\vee$ in $V^\vee$. As $\xi(e_i)=\lambda_i$ we get that 
\[\|\xi\|\geqslant |\lambda_i|/\|e_i\|\geqslant\alpha |\lambda_i|\cdot\|e_i^{\vee}\|.\]
This completes the proof of the proposition.
\end{proof}
Let $(V_1,\varphi_1)$ and $(V_2,\varphi_2)$ be two ultra-normed finite-dimensional $k$-vector spaces. We can then 
identify $V_1\otimes V_2$ with $\Hom_k(V_1^\vee,V_2)$ and equip it with the operator norm. This ultrametric norm on 
$V_1\otimes V_2$, is called the \emph{tensor} norm and is denoted by $\varphi_1\otimes\varphi_2$. 
We can construct a tensor norm for a tensor product of multiple ultra-normed spaces recursively 
(cf. \cite[remarque 3.8]{Gaudron08}).  
\begin{prop}\label{Pro:alphatenso}
Let $V$ and $W$ be two finite dimensionel ultra-normed $k$-vector spaces and let $\alpha$ be a real number in $]0,1]$. 
If $(s_i)_{i=1}^n$ and $(t_j)_{j=1}^m$ are $\alpha$-orthogonal bases of $V$ and $W$ respectively then 
$(s_i\otimes t_j)_{\begin{subarray}{c}1\leqslant i\leqslant n\\
1\leqslant j\leqslant m\end{subarray}}$ is an $\alpha^2$-orthogonal basis of  $V\otimes W$.
\end{prop}
\begin{proof}
Consider $(a_{ij})_{\begin{subarray}{c}1\leqslant i\leqslant n\\
1\leqslant j\leqslant m\end{subarray}}$ an $n\times m$ matrix with coefficients in $k$. Let $T$ be the tensor 
$\sum_{i=1}^n\sum_{j=1}^m a_{ij}s_i\otimes t_j$. We seek a lower bound for the operator norm of $T$, considered as a 
$k$-linear map from 
$V^\vee$ to $W$. Let $(s_i^\vee)_{i=1}^n$ be the dual basis of $(s_i)_{i=1}^n$. By Proposition \ref{Pro:alphaorthogonale} 
this is a $\alpha$-orthogonal basis of $V^\vee$. For any $i\in\{1,\ldots,n\}$ we have that 
\[\|T(s_i^\vee)\|=\bigg\|\sum_{j=1}^ma_{ij}t_j\bigg\|\geqslant
\alpha\max_{j\in\{1,\ldots,m\}}|a_{ij}|\cdot\|t_j\|.\]
This proves that
\[\|T\|\geqslant \|s_i^\vee\|^{-1}\alpha\max_{j\in\{1,\ldots,m\}}|a_{ij}|\cdot\|t_j\|\geqslant\alpha^2\max_{{j\in\{1,\ldots,m\}}}|a_{ij}|\cdot\|s_i\|\cdot\|t_j\|,\]
where the second inequality follows from \eqref{Equ:normeduale}.
This completes the proof of the proposition.
\end{proof}

For any $k$-vector space $V$ of finite dimension $r$ equipped with an ultrametric norm $\|.\|$ we equip 
$\det(V)=\Lambda^r(V)$ with the quotient norm induced by the canonical map $V^{\otimes r}\rightarrow\Lambda^r(V)$. The 
ultrametric Hadamard inequality implies that for any basis $(e_1,\ldots,e_r)$ of $V$ we have that
\begin{equation}\label{Equ:Hadamard}
\|e_1\wedge\cdots\wedge e_r\|\leqslant \|e_1\otimes\cdots\otimes e_r\|=\prod_{i=1}^r\|e_i\|.
\end{equation}
Equality holds when the basis $(e_1,\ldots,e_r)$ is orthogonal. If the basis is $\alpha$-orthogonal then we have that
\begin{equation}\label{Equ:Hadamardinverse}
\|e_1\wedge\cdots\wedge e_r\|\geqslant \alpha^r\prod_{i=1}^r\|e_i\|.\end{equation}
This follows from\footnote{By induction, Proposition \ref{Pro:alphatenso} implies that 
$(e_{\sigma(1)}\otimes\cdots\otimes e_{\sigma(r)})_{1\leqslant\sigma(1),\ldots,\sigma(r)\leqslant r}$ is an 
$\alpha^r$-basis of $V^{\otimes r}$. If $\xi=\sum_\sigma a_\sigma e_{\sigma(1)}\otimes\cdots\otimes e_{\sigma(r)}$ is
an element of $V^{\otimes n}$ whose image in $\Lambda^rV$ is $e_1\wedge\cdots\wedge e_r$ then 
$\sum_{\sigma\in\mathfrak S_r}a_\sigma(-1)^{\mathrm{sgn}(\sigma)}=1$ where $\mathfrak S_r$ is the $r$-th symmetric 
group. There is at least one $\sigma\in\mathfrak S_r$ such that $|a_\sigma|\geqslant 1 $. It follows that 
$\|\xi\|\geqslant \alpha^r\prod_{i=1}^r\|e_i\|$.}proposition \ref{Pro:alphatenso}.

\subsection{Arakelov degree and the Harder-Narasimhan polygon.}\label{Sec:degrehn}
Let $\mathcal C_k$ be the class of finite dimensional $k$-vector spaces equipped with two ultrametric norms. 
If $\overline V=(V,\|.\|_{\varphi},\|.\|_\psi)$ is an element of $\mathcal C_k$ we let $\widehat{\deg}(\overline V)$ be 
the following number
\[-\ln\|s_1\wedge\cdots\wedge s_r\|_{\varphi}+\ln\|s_1\wedge\cdots\wedge s_r\|_\psi,\]
where $(s_1,\ldots,s_r)$ is an arbitrary $k$-basis of $V$. This construction does not depend on the choice of 
$(s_1,\ldots,s_r)$ by the trivial product formula
\[\forall\,a\in k^{\times},\quad -\ln|a|+\ln|a|=0.\]
If $V$ is non-trivial we let $\widehat{\mu}(\overline V)$ be the quotient $\widehat{\deg}(\overline V)/\rang(V)$.
It follows from Proposition \ref{Pro:existenceepsorth}, Hadamard's inequality \eqref{Equ:Hadamard} and the inverse
Hadamard inequality \eqref{Equ:Hadamardinverse} that for any flag of subspaces of $V$
\[0=V_0\subsetneq V_1\subsetneq \ldots\subsetneq V_n=V\]
we have that
\begin{equation}\label{Equ:additivitedeg}\widehat{\deg}(\overline V)=\sum_{i=1}^n\widehat{\deg}(\overline V_{\!i}/\overline V_{\!i-1}),\end{equation}
using the subquotient norms.

We let $\widetilde P_{\overline V}$ be the function on $[0,r]$ whose graph is the concave 
upper bound of the set of points in 
$\mathbb R^2$ of the form $(\rang(W),\widehat{\deg}(\overline W))$, where $W$ runs over the set of subspaces of $V$. 
This is a concave function which is affine on each piece $[i-1,i]$ ($i\in\{1,\ldots,r\}$). For any $i\in\{1,\ldots,r\}$
we let $\widehat{\mu}_i(\overline V)$ be the slope of this function on $[i-1,i]$. We introduce a normalised version of
$\widetilde P_{\overline V}$ by setting
\[\forall\,t\in[0,1],\quad P_{\overline V}(t)=\frac{1}{\rang(V)}\widetilde P_{\overline V}(t\rang(V)).\]
Let $Z_{\overline V}$ be a random variable whose probability law is given by 
\[\frac{1}{\rang(V)}\sum_{
i=1}^{\rang(V)}\delta_{\widehat{\mu}_i(\overline V)}.\]
The following equality also holds in the non-archimedean case
\begin{equation}
\widehat{\mu}(\overline V)=P_{\overline V}(1)=\mathbb E[Z_{\overline V}].
\end{equation}
The results of \S\ref{Subsec:polygonenorme} -- in particular Proposition \ref{Pro:comparaisonnorme} and Corollary 
\ref{Cor:compraisonpolygone} -- still hold for members of $\mathcal C_k$ (and their proofs are similar). We now 
summarise these properties:
\begin{prop}
Let $(V,\varphi,\psi)$ be an element of $\mathcal C_k$. Let $\varphi'$ and $\psi'$ be two ultrametric norms on $V$.
\begin{enumerate}[(1)]
\item If $\psi'\leqslant\psi$ then $\widetilde P_{(V,\varphi,\psi')}\leqslant\widetilde P_{(V,\varphi,\psi)}$.
\item If $\varphi'\geqslant\varphi$ then $\widetilde P_{(V,\varphi',\psi)}\leqslant \widetilde P_{(V,\varphi',\psi)}$.
\item In general we have that
\[\forall\,t\in[0,\rang(V)],\quad\big|\widetilde P_{(V,\varphi,\psi)}(t)-
\widetilde P_{(V,\varphi',\psi')}(t)\big|\leqslant\big(d(\varphi,\varphi')+
d(\psi,\psi')\big)t.\]
\end{enumerate}
\end{prop}
The following proposition can be seen as an ultrametric analogue of the Courant-Fischer theorem.
\begin{prop}\label{Pro:courantfischerultrametrique}
Let $\overline V=(V,\varphi,\psi)$ be an element of $\mathcal C_k$ of dimension $r\geqslant 1$. For any 
$i\in\{1,\ldots,r\}$ we have that
\begin{equation}
\widetilde P_{\overline V}(i)=\sup_{\begin{subarray}{c}
W\subset V\\
\rang(W)=i
\end{subarray}}\widehat{\deg}(\overline W).
\end{equation}
\end{prop}

\begin{proof}
We prove by induction on $r$ that for any $\alpha\in\,]0,1[\,$ there is a basis in $V$ which is $\alpha$-orthogonal for
both $\varphi$ and $\psi$. The case $r=1$ is trivial. Suppose that this statement has been proved for all elements of 
$\mathcal C_k$ of dimension $<r$. We choose $e_1\in V\setminus\{0\}$ such that 
\begin{equation}\label{Equ:premier}\frac{\|e_1\|_{\psi}}{\|e_1\|_{\varphi}}\leqslant (\sqrt[4]{\alpha})^{-1}\inf_{0\neq x\in V}\frac{\|x\|_{\psi}}{\|x\|_{\varphi}}.\end{equation}
By Proposition \ref{Pro:existenceepsorth} there is a subspace $W\subset V$ of dimension $r-1$ such that
\begin{equation}\label{Equ:minorationvarphi}
\forall\,y\in W,\quad\|e_1+y\|_{\varphi}\geqslant\sqrt[4]{\alpha}\max(\|e_1\|_\varphi,\|y\|_{\varphi}).\end{equation}
By \eqref{Equ:premier} and \eqref{Equ:minorationvarphi} we have that
\[\forall\,y\in W,\quad\|e_1+y\|_{\psi}\geqslant\sqrt[4]{\alpha}\cdot\|e_1+y\|_{\varphi}\cdot\frac{\|e_1\|_{\psi}}{\|e_1\|_{\varphi}}\geqslant\sqrt{\alpha}\cdot\|e_1\|_{\psi}.\]
Moreover, as $\psi$ is ultrametric, if $\|y\|_{\psi}>\|e_1\|_{\psi}$ then $\|e_1+y\|_\psi=\|y\|_\psi$. It follows that
\begin{equation}\label{Equ:minorationpsi}\|e_1+y\|_{\psi}\geqslant\sqrt{\alpha}\max(\|e_1\|_\psi,\|y\|_\psi).\end{equation}
By the induction hypothesis, there is a basis $(e_2,\ldots,e_r)$ of $W$ which is $\sqrt{\alpha}$-orthogonal for both 
$\varphi$ and $\psi$. Let us prove that $(e_1,\ldots,e_r)$ is $\alpha$-orthogonal for both $\varphi$ and $\psi$. Let 
$(\lambda_1,\ldots,\lambda_r)$ be an element of $k^r$ with $\lambda_1\neq 0$. By \eqref{Equ:minorationvarphi} we have
that
\[\begin{split}&\quad\;\|\lambda_1e_1+\cdots+\lambda_re_r\|_\varphi\geqslant\sqrt[4]{\alpha}\cdot
\max(|\lambda_1|\cdot\|e_1\|_{\varphi},\|\lambda_2e_2+\cdots+\lambda_re_r\|_{\varphi})\\
&\geqslant \sqrt[4]{\alpha}\cdot
\max(|\lambda_1|\cdot\|e_1\|_{\varphi},\sqrt{\alpha}\cdot\max(|\lambda_2|\cdot\|e_2\|_{\varphi},\cdots,|\lambda_r|\cdot\|e_r\|_{\varphi}))\\
&\geqslant \alpha\max(|\lambda_1|\cdot\|e_1\|_{\varphi},\ldots,|\lambda_r|\cdot\|e_r\|_{\varphi}),
\end{split}\]
where the second inequality comes from the fact that $(e_2,\ldots,e_r)$ is $\sqrt{\alpha}$-orthogonal for $W$ for the 
norm $\|.\|_{\varphi}$. This inequality also holds when $\lambda_1=0$ (here we use directly the fact that 
$(e_2,\ldots,e_r)$ is an $\sqrt{\alpha}$-orthogonal basis).
Similarly, it follows from \eqref{Equ:minorationpsi} that $(e_1,\ldots,e_r)$ is an $\alpha$-orthogonal basis for $V$ 
with respect to $\|.\|_{\psi}$.

We now prove the proposition. We deal first with the case where there is a basis $(s_1,\ldots,s_r)$ which is 
orthogonal for both $\varphi$ and $\psi$. For any $i\in\{1,\ldots,r\}$ let $a_i$ be the logarithm of the ratio 
$\|s_i\|_{\psi}/\|s_i\|_{\varphi}$. Without loss of generality we may assume that 
$a_1\geqslant a_2\geqslant\ldots\geqslant a_r$.  For any integer $m\in\{1,\ldots,r\}$ the vectors 
$s_{i_1}\wedge\cdots\wedge s_{i_m}$ ($1\leqslant i_1<\ldots<i_m\leqslant r$) form a basis for $\Lambda^m(V)$ which is 
orthogonal for both $\varphi$ abd $\psi$. For any $m$-dimensional subspace $W\subset V$, writing a non-zero element
of  $\Lambda^mW$ as a sum of elements of the form $s_{i_1}\wedge\cdots\wedge s_{i_m}$ enables us to prove that 
\[\widehat{\deg}(\overline W)\leqslant a_1+\cdots+a_m,\]
and equality is achieved when $W$ is generated by the $s_1,\ldots,s_m$.

In the general case, the above proposition enables us to construct a sequence of objects $(V,\varphi_n,\psi_n)$ 
($n\in\mathbb N$) in $\mathcal C_k$ such that\footnote{As in the complex case, for any pair $(\eta,\eta')$ of 
ultrametric norms on $V$ we set
\[d(\eta,\eta')=\sup_{0\neq x\in V}\Big|\ln\|x\|_{\eta}-\ln\|x\|_{\eta'}\Big|.\]} 
\[\lim_{n\rightarrow+\infty}d(\varphi_n,\varphi)+d(\psi_n,\psi)=0\]
and for any $n$ there is a basis of $V$ which is orthogonal for both $\varphi_n$ and $\psi_n$. On the one hand, 
\[\widetilde P_{(V,\varphi_n,\psi_n)}(i)=\sup_{\begin{subarray}{c}W\subset V\\\rang(W)=i\end{subarray}}
\widehat{\deg}(W,\varphi_n,\psi_n),\]
and on the other, we have that
\[\big|\widetilde P_{(V,\varphi_n,\psi_n)}(t)-\widetilde P_{(V,\varphi,\psi)}(t)\big|\leqslant (d(\varphi_n,\varphi)+d(\psi_n,\psi))t,\]
and for any subspace $W\subset V$ we have that 
\[\big|\widehat{\deg}(W,\varphi_n,\psi_n)-\widehat{\deg}(W,\varphi,\psi)\big|\leqslant(d(\varphi_n,\varphi)+d(\psi_n,\psi))\rang(W).\]
As $n\rightarrow+\infty$ we obtain the desired result.
\end{proof}
\subsection{Truncation}\label{Subsec:troncaturepadique} The results in \S\ref{SubSec:Troncature} are still valid
for elements of
$\mathcal C_k$. The proof is simpler because we only consider ultrametric norms. Let $V$ be a finite dimensional 
$k$-vector space and let $\varphi$ be an ultrametric norm on $V$. For any real number $a$ let $\varphi(a)$ be the norm
on $V$ such that
\[\forall\,x\in V,\quad \|x\|_{\varphi(a)}=\mathrm{e}^a\|x\|_{\varphi}.\]
If $\varphi$ and $\psi$ are two ultrametric norms on $V$ we let $\varphi\vee\psi$ be the norm on $V$ such that
\[\forall\,x\in V,\quad \|x\|_{\varphi\vee\psi}=\max(\|x\|_{\varphi},\|x\|_\psi).\]
This is also an ultrametric norm on $V$.
\begin{prop}\label{Pro:troncaturepadique}
Let $\overline V=(V,\varphi,\psi)$ be a non-trivial element of $\mathcal C_k$ and let $a$ be a real number. We have 
that
\[\widehat{\deg}(V,\varphi,\psi\vee\varphi(a))=\sum_{i=1}^{\rang(V)}\max(\widehat{\mu}_i(\overline V),a)\]
\end{prop}
\begin{proof}
We start with the case where $V$ has a basis $(e_1,\ldots,e_r)$ which is orthogonal for both $\varphi$ and $\psi$. Without loss
of generality, we have that
\[\ln\frac{\|e_i\|_\psi}{\|e_i\|_{\varphi}}=\widehat{\mu}_i(\overline V).\] 
The basis $(e_1,\ldots,e_r)$ is also orthogonal for the norm $\psi\vee\varphi(a)$ and we have that
\[\ln\frac{\|e_i\|_\psi}{\|e_i\|_{\varphi}}=\max(\widehat{\mu}_i(\overline V),a).\]
The result follows. In general, we can approximate $(\varphi,\psi)$ by pairs of norms $((\varphi_n,\psi_n))_{n\geqslant 1}$ such that for every $n$ the vector space $V$ has a basis orthogonal for both $\varphi_n$ and $\psi_n$ (cf. the proof of Proposition \ref{Pro:courantfischerultrametrique}). Passing to the limit $n\rightarrow+\infty$, we get our result.
\end{proof}

\section{Equilibrium energy.}\label{Sec:energiealequilibre}
In this section we fix a field $k$, which is either $\mathbb C$ with the usual absolute value, or a 
complete field equipped with a non-archimedean absolute value. If $k=\mathbb C$ then we will denote by 
$\mathcal C_k$ the class $\mathcal C$ defined in \S\ref{Subsec:polygonenorme}.
\subsection{Monomial bases and the Okounkov semi-group.} In this subsection we recall the construction of
the Okounkov semi-group of a graded linear series. We refer the reader to \cite{Okounkov96,Lazarsfeld_Mustata08,
Boucksom_Bourbaki, Kaveh_Khovanskii} for more details. 

Consider an integral projective scheme $X$ of dimension $d\geqslant 1$ defined over the field $k$. We assume that the
scheme $X$ has a regular rational point $x$: the local ring $\mathcal O_{X,x}$ is then a regular local ring of 
dimension $d$. We fix a regular sequence $(z_1,\ldots,z_d)$ in its maximal ideal $\mathfrak m_x$. The formal completion
of $\mathcal O_{X,x}$ with respect to the maximal ideal $\mathfrak m_x$ is isomorphic to the algebra of formal series
in the parameters $z_1,\ldots,z_d$ (cf. \cite[proposition 10.16]{Eise}).

If we choose a monomial ordering\footnote{This is a total order $\leqslant$ on $\mathbb N^d$ such that 
$0\leqslant\alpha$ for any $\alpha\in\mathbb N^d$ and $\alpha\leqslant\alpha'$ implies $\alpha+\beta\leqslant\alpha'+\beta$ for any $\alpha,\alpha'$ and $\beta$ in $\mathbb N^d$.} $\leqslant$ on $\mathbb N^d$ we obtain a decreasing
$\mathbb N^d$-filtration (called the \emph{Okounkov filtration}) $\mathcal F$ on $\widehat{\mathcal O}_{X,x}$ such that
$\mathcal F^{\alpha}(\widehat{\mathcal O}_{X,x})$ is the ideal generated by monomials of the form $z^{\beta}$ such that 
$\beta\geqslant\alpha$. This filtration is multiplicative: we have that
\[\mathcal F^\alpha(\widehat{\mathcal O}_{X,x})\mathcal F^\beta(\widehat{\mathcal O}_{X,x})\subset
\mathcal F^{\alpha+\beta}(\widehat{\mathcal O}_{X,x}).\]

The filtration $\mathcal F$ induces by grading an $\mathbb N^d$-graded algebra $\mathrm{gr}(\mathcal O_{X,x})$ which is 
isomorphic\footnote{This follows from the fact that $\mathcal O_{X,x}$ is dense in $\widehat{\mathcal O}_{X,x}$.} to
 $k[z_1,\ldots,z_d]$. In particular, for any $\alpha\in\mathbb N^d$, $\mathrm{gr}^{\alpha}(\mathcal O_{X,x})$ is a 
rank-one vector space on $k$.

If $L$ is an inversible $\mathcal O_X$-module then on taking a local trivialisation in a neighbourhood of $x$ we can identify $L_x$
with $\mathcal O_{X,x}$. The filtration $\mathcal F$ then induces a decreasing $\mathbb N^d$-filtration 
on $H^0(X,L)$ which is independent of the choice of trivialisation. For any $s\in H^0(X,L)$ we denote by 
$\mathrm{ord}(s)$
the upper bound of the set of $\alpha\in\mathbb N^d$ such that $s\in\mathcal F^\alpha H^0(X,L)$. We have that 
\[\forall\,s,s'\in H^0(X,L),\quad\mathrm{ord}(s+s')\geqslant\min\big(\mathrm{ord}(s),\mathrm{ord}(s')\big).\]
Moreover, for any $s\in H^0(X,L)$ and any $a\in k^{\times}$ we have that $\mathrm{ord}(s)=\mathrm{ord}(as)$.

Let $L$ be an invertible $\mathcal O_X$-module. We let $V_\sbullet(L)$ be the graded ring $\bigoplus_{n\geqslant 0}
H^0(X,nL)$ (with the additive notation for the tensor product of invertible sheaves). By a  \emph{graded linear system}
of $L$ we mean a graded subalgebra of $V_\sbullet(L)$. Any graded linear system  $V_\sbullet$ of $L$ can be identified, on choosing 
a local trivialisation of $L$ around $x$, with a graded subalgebra of the algebra of 
polynomials $\mathcal O_{X,x}[T]$. The filtration $\mathcal F$ induces a decreasing $\mathbb N^d$-filtration on each 
homogeneous piece $V_n$. We denote by $\mathrm{gr}(V_\sbullet)$ the $\mathbb N^{d+1}$-graded $k$-algebra induced by this
filtration. This is an $\mathbb N^{d+1}$-graded subalgebra of 
$\mathrm{gr}(\mathcal O_{X,x})[T]\cong k[z_1,\ldots,z_d,T]$. In particular, the elements $(n,\alpha)\in\mathbb N^{d+1}$ 
such that $\mathrm{gr}^{(n,\alpha)}(V_\sbullet)\neq \{0\}$ form a sub-semigroup of $\mathbb N^{d+1}$ which we denote by 
$\Gamma(V_{\sbullet})$. For any $n\in\mathbb N$ we denote by $\Gamma(V_n)$ the subset of $\mathbb N^d$ of elements 
$\alpha$ such that $(n,\alpha)\in\Gamma(V_\sbullet)$.
\subsection{Monomial norms.}\label{SubSec:normesmonomiales}
As above, we consider an integral projective scheme $X$ of dimension $d\geqslant 1$ over $\Spec k$. We fix a regular 
rational point $x\in X(k)$ (it is assumed that such a point exists), a system of parameters $z=(z_1,\ldots,z_d)$ at $x$ and
a monomial order on $\mathbb N^d$. Let $L$ be an invertible $\mathcal O_X$-module and let $V_\sbullet$ be a graded linear
system of $L$. We assume that every $k$ vector space $V_n$ is equipped with two norms $\varphi_n$ and $\psi_n$, which 
are ultrametric if $k$ is non-archimedean. We assume moreover that these norms are submultiplicative - ie. that for any
$(n,m)\in\mathbb N^2$ and any $(s_n,s_m)\in V_n\times V_m$ we have that
\begin{equation}\label{Equ:sousmultiplicative}\|s_n\otimes s_m\|_{\varphi_{n+m}}\leqslant\|s_n\|_{\varphi_n}\cdot\|s_m\|_{\varphi_m},\quad
\|s_n\otimes s_m\|_{\psi_{n+m}}\leqslant\|s_n\|_{\psi_n}\cdot\|s_m\|_{\psi_m}.\end{equation}
In this subsection, we study the asymptotic behaviour of $\widehat{\deg}(V_n,\varphi_n,\psi_n)$. As the 
Harder-Narasimhan filtration is not functorial in $\mathcal C_k$ we cannot study this problem directly using the method
developped in \cite{Boucksom_Chen}. We will avoid this problem by using Okounkov filtrations. The norms $\varphi_n$ and
$\psi_n$ induce quotient norms on each of the sub-quotients $\mathrm{gr}^\alpha(V_n)$ ($\alpha\in\Gamma(V_n)$) which by
abuse of notation we will continue to denote by $\varphi_n$ and $\psi_n$. By the results of previous sections, notably 
\eqref{Equ:sommesousquotient} and \eqref{Equ:additivitedeg}, we have that
\begin{equation*}\bigg|\widehat{\deg}(V_n,\varphi_n,\psi_n)-\sum_{\alpha\in\Gamma(V_n)}\widehat{\deg}(\mathrm{gr}^\alpha(V_n),\varphi_n,\psi_n)\bigg|\leqslant A_k(\rang(V_n)),\end{equation*}
where $A_k(r)=r\ln(r)$ if $k=\mathbb C$ and $A_k(r)=0$ if $k$ is non-archimedean. We deduce that 
\begin{equation}\label{Equ:remenerauquotient}
\lim_{n\rightarrow+\infty}\bigg|\frac{\widehat{\mu}(V_n,\varphi_n,\psi_n)}{n}-\frac{1}{n\#\Gamma(V_n)}\sum_{\alpha\in\Gamma(V_n)}
\widehat{\deg}(\mathrm{gr}^\alpha(V_n),\varphi_n,\psi_n)\bigg|=0.
\end{equation}
Moreover, if we equip the space
\[\mathrm{gr}(V_n)=\bigoplus_{\alpha\in\Gamma(V_n)}\mathrm{gr}^\alpha(V_n)\]
with norms $\hat{\varphi}_n$ and $\hat{\psi}_n$ such that the rank $1$ subspaces $\mathrm{gr}^\alpha(V_n)$ are 
orthogonal with the induced (subquotient) norms $\varphi_n$ and $\psi_n$ respectively then we have that
\[\sum_{\alpha\in\Gamma(V_n)}\widehat{\deg}(\mathrm{gr}^\alpha(V_n),\varphi_n,\psi_n)=\widehat{\deg}(\mathrm{gr}(V_n),\hat{\varphi}_n,\hat{\psi}_n).\]
The problem then reduces to the study of the $k$-algebra of the semi-group $\Gamma(V_\sbullet)$ (isomorphic to 
$\mathrm{gr}(V_\sbullet)$) with the appropriate norms. As the semi-group $\Gamma(V_\sbullet)$ is a multiplicative basis
for the algebra $k[\Gamma(V_\sbullet)]$ we can construct a new norm $\eta_n$ on each space $\mathrm{gr}(V_n)$ which is
multiplicative: for any $\gamma\in\Gamma(V_\sbullet)$ we let $s_{\gamma}$ be the canonical image of 
$\gamma\in\Gamma(V_\sbullet)$ in the algebra $k[\Gamma(V_\sbullet)]$ and we equip 
\[\mathrm{gr}(V_n)=\bigoplus_{\alpha\in\Gamma(V_n)}k s_{(n,\alpha)}\]
with the norm $\eta_n$ such that the vectors $s_{n,\alpha}$ are orthogonal of norm $1$. Using these auxillary norms we 
can write the degree  $\widehat{\deg}(\mathrm{gr}(V_n),\hat{\varphi}_n,\hat{\psi}_n)$ as a difference
\[\widehat{\deg}(\mathrm{gr}(V_n),\hat{\varphi}_n,\eta_n)-\widehat{\deg}(\mathrm{gr}(V_n),\hat{\psi}_n,\eta_n),\]
or alternatively
\[\sum_{\alpha\in\Gamma(V_n)}\widehat{\deg}(\mathrm{gr}^\alpha(V_n),\varphi_n,\eta_n)-\sum_{\alpha\in\Gamma(V_n)}\widehat{\deg}(\mathrm{gr}^\alpha(V_n),\psi_n,\eta_n).\]
It is easy to see that the real valued functions $(n,\alpha)\mapsto\widehat{\deg}(\mathrm{gr}^\alpha(V_n),\varphi_n,
\eta_n)$ and $(n,\alpha)\mapsto\widehat{\deg}(\mathrm{gr}^\alpha(V_n),\varphi_n,\eta_n)$ defined on $\Gamma(V_\sbullet)$  
are superadditive, so their asymptotic behaviour can be studied using the methods developped in \cite{Boucksom_Chen}.
\subsection{Limit theorem}\label{Subsec:thmlimite} In this subsection we fix an integer $d\geqslant 1$ and a 
sub-semigroup $\Gamma$ in $\mathbb N^{d+1}$. For any integer $n\in\mathbb N$ we denote by $\Gamma_n$ the set 
$\{\alpha\in\mathbb N^d\,|\,(n,\alpha)\in\Gamma\}$. We suppose that the semi-group $\Gamma$ verifies the following 
conditions (cf. \cite[\S2.1]{Lazarsfeld_Mustata08})~:
\begin{enumerate}[(a)]
\item $\Gamma_0=\{\boldsymbol{0}\}$,
\item there is a finite subset $B$ in $\{1\}\times\mathbb N^d$ such that $\Gamma$ is contained in the sub-monoid of 
$\mathbb N^{d+1}$ generated by $B$,
\item the group $\mathbb Z^{d+1}$ is generated by $\Gamma$.
\end{enumerate}
We let $\Sigma(\Gamma)$ be the (closed) convex cone in $\mathbb R^{d+1}$ generated by $\Gamma$. Under the above 
conditions the projection of $\Sigma\cap(\{1\}\times\mathbb R^d)$ into $\mathbb R^d$ is a convex body in $\mathbb R^d$,
denoted $\Delta(\Gamma)$. Moreover, we have that 
\[\lim_{n\rightarrow+\infty}\frac{\#\Gamma_n}{n^d}=\mathrm{vol}(\Delta(\Gamma)),\]
where $\mathrm{vol}(.)$ is Lesbesgue measure on $\mathbb R^d$  (cf. \cite[proposition 2.1]{Lazarsfeld_Mustata08}).

We say that a function $\Phi:\Gamma\rightarrow\mathbb R$ is \emph{superadditive} if 
$\Phi(\gamma+\gamma')\geqslant\Phi(\gamma)+\Phi(\gamma')$. In what follows, we study the asymptotic properties of
super-additive functions.
\begin{lemm}
Let $\Phi$ be a superadditive function defined on $\Gamma$ such that 
$\Phi(0,0)=0$.
\begin{enumerate}[(1)]
\item For any real number $t$ the set $\Gamma_{\Phi}^t:=\{(n,\alpha)\in\Gamma\,|\,\Phi(n,\alpha)\geqslant nt\}$ is a sub-semigroup of $\Gamma$.
\item If $t\in\mathbb R$ is a real number such that
\[t<\lim_{n\rightarrow+\infty}\sup_{\alpha\in\Gamma_n}\frac{1}{n}\Phi(n,\alpha),\]
then $\Gamma_{\Phi}^t$ satisfies conditions (a)--(c) above.
\end{enumerate}
\end{lemm}
\begin{proof}
(1) As $\Phi$ is superadditive, for any $(n,\alpha)$ and $(m,\beta)$ in $\Gamma_{\Phi}^t$ we have that 
\[\Phi(n+m,\alpha+\beta)\geqslant\Phi(n,\alpha)+\Phi(m,\beta)\geqslant nt+mt=(n+m)t,\]
and hence $(n+m,\alpha+\beta)\in\Gamma_{\Phi}^t$.

(2) It is easy to check that (a) and (b) are satisfied by $\Gamma_{\Phi}^t$. We now prove (c). Let $A$ be a finite 
subset of $\Gamma$ generating $\mathbb Z^{d+1}$ as a group. By hypothesis, there exists a $\varepsilon>0$ and a 
$\gamma=(m,\beta)\in\Gamma$ such that $\Phi(m,\beta)\geqslant(t+\varepsilon)m$. It follows that for any 
$(n,\alpha)\in\Gamma$, we have that
\[\frac{\Phi(km+n,k\beta+\alpha)}{km+n}\geqslant\frac{\Phi(n,\alpha)+k\Phi(m,\beta)}{km+n}\geqslant\frac{\Phi(n,\alpha)+km(t+\varepsilon)}{km+n}\geqslant t\]
for large enough $k$. There therefore exists a $k_0\geqslant 1$ such that $k\gamma+\xi\in\Gamma_{\Phi}^t$ for any 
$\xi\in A$ and $k\geqslant k_0$, so $\Gamma_{\Phi}^t$ generates $\mathbb Z^{d+1}$ as a group.
\end{proof}
\begin{rema}\label{Rem:foncrepartition}
The superadditivity of $\Phi$ implies that
\[\Delta(\Gamma_\Phi^{\varepsilon t_1+(1-\varepsilon)t_2})\supset \varepsilon\Delta(\Gamma_\Phi^{t_1})+(1-\varepsilon)\Delta(\Gamma_\Phi^{t_2}).\] By the Brunn-Minkowski theorem, the function $t\mapsto\mathrm{vol}(\Delta(\Gamma_\Phi^t))^{1/d}$ is
concave on $]-\infty,\theta[\,$, where 
\[\theta=\lim_{n\rightarrow+\infty}\sup_{\alpha\in\Gamma_n}\frac{1}{n}\Phi(n,\alpha),\] so it is continuous on this interval. 
Moreover, as the set (dense in $\Delta(\Gamma)$) \[\{\alpha/n\,:\,(n,\alpha)\in\Gamma,\,n\geqslant 1\}\] 
is contained in $\bigcup_{t\in\mathbb R}\Delta(\Gamma_{\Phi}^t)$, we get that 
\[\mathrm{vol}(\Delta(\Gamma))=\lim_{t\rightarrow-\infty}\mathrm{vol}(\Delta(\Gamma_\Phi^t)).\] 
\end{rema}
The following result is a limit theorem for superadditive functions defined on $\Gamma$. It is a natural generalisation
of \cite[Theorem 1.11]{Boucksom_Chen}
\begin{theo}\label{Thm:thmdelimite}
Let $\Phi:\Gamma\rightarrow\mathbb R$ be a superadditive function such that
\begin{equation}\label{Equ:seuil}\theta:=\lim_{n\rightarrow+\infty}\sup_{\alpha\in\Gamma_n}\frac{1}{n}\Phi(n,\alpha)<+\infty.\end{equation}
For any integer $n\geqslant 1$, we consider $Z_n=\Phi(n,.)$ as a uniformly distributed random variable on $\Gamma_n$. 
The sequence of random variables $\big(Z_n/n\big)_{n\geqslant 1}$ then converges in 
law\footnote{We say that a sequence of random
variables $(Z_n)_{n\geqslant 1}$ converges in law to a random variable $Z$ if the law of $Z_n$ converges weakly to that
of $Z$, i.e., for any continuous bounded function $h$ on $\mathbb R$ we have that 
$\displaystyle\lim_{n\rightarrow+\infty}\mathbb E[h(Z_n)]=\mathbb E[Z]$, or equivalently, the probability function of 
$Z_n$ converges to that of $Z$ at any point $x\in\mathbb R$ such that $\mathbb P(Z=x)=0$.} to a limit random variable 
$Z$ whose law is given by
\[\mathbb P(Z\geqslant t)=\frac{\mathrm{vol}(\Delta(\Gamma_{\Phi}^t))}{
\mathrm{vol}(\Delta(\Gamma))},\quad t\neq\theta.\]
\end{theo}
\begin{proof} By Remark \ref{Rem:foncrepartition} the function $F$ defined on $t\in\mathbb R\setminus\{\theta\}$ by 
$F(t):=\mathrm{vol}(\Delta(\Gamma_\Phi^t))/\mathrm{vol}(\Delta(\Gamma))$ is decreasing and continuous and
$\displaystyle\lim_{t\rightarrow-\infty}F(t)=1$. 
Moreover, condition \eqref{Equ:seuil} implies that $F(t)=0$ for large enough positive $t$ and it follows that
if we extend the domain of definition of $F$ to $\mathbb R$ by taking $F(\theta)$ to be the limit of $F(t)$ as $t$ 
tends to $\theta$ from the left we get a (left continuous) probability function on $\mathbb R$. For any integer 
$n\geqslant 1$ and any real number $t$ we have that 
\[\mathbb P(Z_n\geqslant t)=\frac{\#\Gamma_{\Phi,n}^t}{\#\Gamma_n},\]
where $\Gamma_{\Phi,n}^t$ is the set of all $\alpha\in\mathbb N^d$ such that $(n,\alpha)\in\Gamma_{\Phi}^t$.
By the previous lemma we have that
\begin{equation}\label{Equ:convergence}\lim_{n\rightarrow+\infty}\mathbb P(Z_n\geqslant t)=F(t)\end{equation}
for any $t<\theta$. Moreover, if $t>\theta$ then $\Gamma_{\Phi,n}^t$ is empty for any $n\geqslant 1$ and 
$\Delta(\Gamma_{\Phi}^t)$ is also empty, so equation \eqref{Equ:convergence} also holds for $t>\theta$. Finally, if
the function $F$ is continuous at $\theta$ then since both $t\mapsto\mathbb P(Z_n\geqslant t)$ and $F$ are decreasing
we also have that $\displaystyle\lim_{n\rightarrow+\infty}\mathbb P(Z_n\geqslant\theta)=F(\theta)$. The result follows.
\end{proof}
\begin{rema}\label{Rem:realisationdeZ}
The limit law in the above theorem can also be characterised as the pushforward of Lesbesgue measure on $\Delta(\Gamma)$ by
a function determined by $\Phi$. Let $G_{\Phi}:\Delta(\Gamma)\rightarrow\mathbb R\cup\{-\infty\}$ be the map sending $x$
to $\sup\{t\in\mathbb R\,:\,x\in\Delta(\Gamma_\Phi^t)\}$. This is a real concave function on\footnote{The set 
$\bigcup_{t\in\mathbb R}\Delta(\Gamma_\Phi^t)$ is convex and its volume is equal to $\mathrm{vol}(\Delta(\Gamma))$ so 
it contains $\Delta(\Gamma)^\circ$.} $\Delta(\Gamma)^\circ$. The function $G_{\Phi}$ is therefore continuous on 
$\Delta(\Gamma)^\circ$. By definition, the limit law is equal to the pushforward of normalised Lesbesgue measure on 
$\Delta(\Gamma)$ by $G_\Phi$. In particular, if $h$ is a continuous bounded function then we have that
\begin{equation}\label{Equ:formulelimit}\lim_{n\rightarrow+\infty}\frac{h(\Phi(n,\alpha)/n)}{\#\Gamma_n}=\frac{1}{\mathrm{vol}(\Delta(\Gamma))}\int_{\Delta(\Gamma)^\circ}G_{\Phi}(x)\,\mathrm{vol}(\mathrm{d}x)\end{equation}
This enables us to realise the random variable $Z$ as the function $G_\Phi$ defined on the convex body $\Delta(\Gamma)$
equipped with normalised Lesbesgue measure.
\end{rema}
In the rest of this section we apply these results to the situation described in \S\ref{SubSec:normesmonomiales}. We
consider an integral projective scheme $X$ of dimension $d\geqslant 1$ defined on a field $k$ and an invertible 
$\mathcal O_X$-module $L$. We also choose a regular rational point (it is assumed that such a point exists) $x\in X(k)$, a 
local system of parameters $(z_1,\ldots,z_d)$ and a monomial order on $\mathbb N^d$. Let $V_\sbullet$ be a graded 
linear system on $L$ whose Okounkov semi-group $\Gamma(V_\sbullet)$ satisfies\footnote{Note that these three conditions
are automatically satified whenever $V_\sbullet$ contains an ample divisor, ie. $V_n\neq\{0\}$ for large enough $n$ and
there is an integer $p\geqslant 1$, an ample $\mathcal O_X$-module $A$ and a non-zero section $s$ of $pL-A$, such that
\[\mathrm{Im}\big(H^0(X,nA)\stackrel{\cdot s^n}{\longrightarrow} H^0(X,npL)\big)\subset V_{np}\]
for any $n\in\mathbb N$, $n\geqslant 1$. We refer the reader to \cite[lemma 2.12]{Lazarsfeld_Mustata08} for a proof.
} conditions (a)-(c) of section 4.3. For any $n\in\mathbb N$ let $V_n$ be equipped with two
norms $\varphi_n$ and $\psi_n$ which are assumed to be ultrametric for non-archimedean $k$. 
\begin{theo}\label{Thm:energieequilibree} Assume the norms $\varphi_n$ and $\psi_n$  satisfy the following conditions:
\begin{enumerate}[(1)]
\item the system of norms $(\varphi_n,\psi_n)_{n\in\mathbb N}$ is submultiplicative (i.e. satisfies 
\eqref{Equ:sousmultiplicative});
\item we have that $d(\varphi_n,\psi_n)=O(n)$ as $n\rightarrow+\infty$;
\item there is a constant $C>0$ such that\footnote{See \S\ref{SubSec:normesmonomiales} for notation.} 
$\inf_{\alpha\in\Gamma(V_n)}\ln\|s_{(n,\alpha)}\|_{\hat\varphi_n}\geqslant -Cn$ for any $n\in\mathbb N$, $n\geqslant 1$. 
\end{enumerate}
Then the sequence $(\frac 1n\widehat{\mu}(V_n,\varphi_n,\psi_n))_{n\geqslant 1}$ converges in $\mathbb R$.
\end{theo}
\begin{proof}
We introduce auxillary monomial norms $\eta_n$ as in \S\ref{SubSec:normesmonomiales}. Let 
$\Phi:\Gamma(V_\sbullet)\rightarrow\mathbb R$ be the function that sends $(n,\alpha)\in\Gamma(V_\sbullet)$ to 
$\widehat{\deg}(\mathrm{gr}^\alpha(V_n),\varphi_n,\eta_n)$. This function is superadditive and condition (3) implies 
that
\[\lim_{n\rightarrow+\infty}\sup_{\alpha\in V_n}\frac 1n\Phi(n,\alpha)<+\infty.\]
Let  $Z_{\Phi,n}=\Phi(n,.)$ be a uniformly distributed random variable on $\Gamma(V_n)$. By Theorem 
\ref{Thm:thmdelimite} the sequence of random variables $(Z_{\Phi,n}/n)_{n\geqslant 1}$ converges in law to a random 
variable $Z_\Phi$ defined on $\Delta(\Gamma(V_\sbullet))$ (as in remark \ref{Rem:realisationdeZ}). Similarly, conditions
(2) and (3) prove that (3) also holds for the norms $\hat{\psi}_n$. Denote by $\Psi:\Gamma(V_\sbullet)\rightarrow
\mathbb R$ the function sending $(n,\alpha)\in\Gamma(V_\sbullet)$ to $\widehat{\deg}(\mathrm{gr}^\alpha(V_n),\psi_n,\eta_n)$ and by $Z_{\Psi,n}=\Psi(n,.)$ the random variable on $\Gamma(V_n)$ such that $n\in\mathbb N$, $n\geqslant 1$. The
sequence of random variables $(Z_{\Psi,n}/n)_{n\geqslant 1}$ then converges in law to a random variable $Z_\Psi$ defined on 
$\Delta(\Gamma(V_\sbullet))$. Moreover, (2) implies that the function $|Z_\Phi-Z_\Psi|$ is bounded on 
$\Delta(\Gamma(V_\sbullet))^\circ$.

By equation \eqref{Equ:remenerauquotient} and the equality
\[\widehat{\deg}(\mathrm{gr}^\alpha(V_n),\varphi_n,\psi_n)=\widehat{\deg}(\mathrm{gr}^{\alpha}(V_n),\varphi_n,\eta_n)-\widehat{\deg}(\mathrm{gr}^{\alpha}(V_n),\psi_n,\eta_n),\]
it will be enough to prove that the sequence  $(\mathbb E[Z_{\Phi,n}/n]-\mathbb 
E[Z_{\Psi,n}/n])_{n\geqslant 1}$ converges in $\mathbb R$. Condition (2) of the theorem implies that the functions 
$\frac 1n|Z_{\Phi,n}-Z_{\Psi,n}|$ ($n\in\mathbb N$) are uniformly bounded. Let $A>0$ be a constant such that
\[\forall\,n\geqslant 1,\quad |Z_{\Phi,n}-Z_{\Psi,n}|\leqslant An.\]
As $(Z_{\Phi,n}/n)_{n\geqslant 1}$ and $(Z_{\Psi,n}/n)_{n\geqslant 1}$ converge in law to $Z_\Phi$ and $Z_\Psi$ respectively, for
any $\varepsilon>0$ there is a  $T_0>0$ and a $n_0\in\mathbb N$ such that
\begin{equation*}\forall\,T\geqslant T_0,\;\forall\,n\geqslant n_0,\quad \mathbb P(Z_{\Phi,n}\leqslant -nT)<\varepsilon\text{ et }\mathbb P(Z_{\Psi,n}\leqslant -nT)<\varepsilon.\end{equation*}It follows that
\begin{equation}\label{Equ:convergenceT}\begin{split}&\quad\;\big|\mathbb E[Z_{\Phi,n}/n]-\mathbb E[Z_{\Psi,n}/n]-\mathbb E[\max(Z_{\Phi,n}/n,-T)]+\mathbb E[\max(Z_{\Psi,n}/n,-T)]\big|\\
&\leqslant 2\varepsilon\mathbb E[|Z_{\Phi,n}/n-Z_{\Psi,n}/n|]\leqslant 2\varepsilon A
\end{split}
\end{equation}
whenever $T\geqslant T_0$ and $n\geqslant n_0$. Moreover, as the random variables $Z_{\Phi,n}/n$ and $Z_{\Psi,n}/n$ are
uniformly bounded above and the sequences $(Z_{\Phi,n}/n)_{n\geqslant 1}$ and $(Z_{\Psi,n}/n)_{n\geqslant 1}$ converge in law
it follows that
\[\lim_{n\rightarrow+\infty}\mathbb E[\max(Z_{\Phi,n}/n,-T)]-\mathbb E[\max(Z_{\Psi,n}/n,-T)]=\mathbb E[\max(Z_\Phi,-T)-\max(Z_\Psi,-T)].\]
Moreover, as the function $|Z_\Phi-Z_\Psi|$ is bounded, the dominated convergence theorem implies that 
\[\lim_{T\rightarrow+\infty}\mathbb E[\max(Z_\Phi,-T)-\max(Z_\Psi,-T)]=\mathbb E[Z_\Phi-Z_{\Psi}].\]
Equation \eqref{Equ:convergenceT} then implies that
\[\limsup_{n\rightarrow+\infty}\big|\mathbb E[Z_{\Phi,n}/n]-\mathbb E[Z_{\Psi,n}/n]-\mathbb E[Z_\Phi-Z_\Psi]\big|\leqslant 2\varepsilon A.\]
As $\varepsilon$ is arbitrary, we get that
\[\lim_{n\rightarrow+\infty}\frac 1n\widehat{\mu}(V_n,\varphi_n,\psi_n)=\mathbb E[Z_\Phi-Z_\Psi].\]
\end{proof}

Condition (3) in the above theorem holds whenever $\varphi_n$ comes from a continuous metric on the invertible 
$\mathcal O_X$-module $L$. This can be proved by considering a monomial order 
$\leqslant$ on $\mathbb N^d$ such that\footnote{
When $\alpha_1+\cdots+\alpha_d=\beta_1+\cdots+\beta_d$ we may use the lexicographic order, for example.} 
$\alpha_1+\cdots+\alpha_d<\beta_1+\cdots+\beta_d$ implies $(\alpha_1,\ldots,\alpha_d)<(\beta_1,\ldots,\beta_d)$.  Let 
$X^{\mathrm{an}}$ be the analytic space associated to the $k$-scheme $X$ (in the Berkovich sense \cite{Berkovich90} if 
$k$ is non-archimedean) and let $L^{\mathrm{an}}$ be the pull-back of $L$ to $X^{\mathrm{an}}$. Let $\mathcal C^0_{X^{\mathrm{an}}}$ be the sheaf of continuous real functions on $X^{\mathrm{an}}$. A \emph{continuous metric} on $L$ is a morphism of set
sheaves, $\|.\|$, from $L^{\mathrm{an}}\otimes\mathcal C^0_{X^{\mathrm{an}}}$ to $\mathcal C^0_{X^{\mathrm{an}}}$ which in every
point $x\in X^{\mathrm{an}}$ induces a norm $\|.\|(x)$ on the fibre $L^{\mathrm{an}}(x)$. Given a continuous metric 
$\varphi$ on $X$ we can equip $H^0(X,L)$ with the supremum norm $\|.\|_{\varphi,\sup}$ such that
\[\forall\,s\in H^0(X,L),\quad \|s\|_{\varphi,\sup}:=\sup_{x\in X^{\mathrm{an}}}\|s\|_{\varphi}(x).\]
For any integer $n\in\mathbb N$ the metric $\varphi$ induces by passage to the tensor product a continuous metric 
$\varphi^{\otimes n}$ on $nL$. Let $\varphi_n$ be the supremum norm on $H^0(X,nL)$ induced by $\varphi^{\otimes n}$ (or its
restriction to $V_n$ by abuse of language): the system of norms $(\varphi_n)_{n\geqslant 0}$ then satisfies condition (3) 
of Theorem \ref{Thm:energieequilibree}. This follows from Schwarz's (complex or non-archimedean) Lemma (cf. 
\cite[pp.205-206]{Chambert}). This gives us the following corollary.

\begin{coro}\label{Cor:thmlimite}
Let $X$ be a projective integral scheme defined over a field $k$ and let $L$ be an invertible $\mathcal O_X$-module 
equipped with two continuous metrics $\varphi$ and $\psi$. Let $V_\sbullet$ be a graded linear system of $L$ such that 
$\Gamma(V_\sbullet)$ satisfies conditions (a)--(c) above. For any integer $n\in\mathbb N$ let $\varphi_n$ and $\psi_n$ 
be the supremum norms on $V_n$ associated to the metrics $\varphi^{\otimes n}$ and $\psi^{\otimes n}$ respectively. The
sequence $(\widehat{\mu}(V_n,\varphi_n,\psi_n)/n)_{n\geqslant 1}$ then converges in $\mathbb R$.
\end{coro}
\begin{proof}
The system of norms $(\varphi_n)_{n\geqslant 0}$ is submultiplicative. If $s$ and $s'$ are elements of $V_n$ and $V_m$ 
respectively we have that 
\[\begin{split}&\quad\;\|s\otimes s'\|_{\varphi_{n+m}}=\sup_{x\in X^{\mathrm{an}}}\|s\otimes s'\|_{\varphi^{\otimes(n+m)}}(x)\\&\leqslant \Big(\sup_{x\in X^{\mathrm{an}}}\|s\|_{\varphi^{\otimes n}}(x)\Big)\cdot\Big(\sup_{x\in X^{\mathrm{an}}}\|s'\|_{\varphi^{\otimes m}}(x)\Big)=\|s\|_{\varphi_n}\cdot\|s'\|_{\varphi_m}.
\end{split}\]
Similarly, the system of norms $(\psi_n)_{n\geqslant 0}$ is also submultiplicative.
Moreover, as the topological space $X^{\mathrm{an}}$ is compact, we have that
\[\sup_{x\in X^{\mathrm{an}}}d(\|.\|_{\varphi}(x),\|.\|_{\psi}(x))<+\infty.\] and it follows that 
$d(\varphi_n,\psi_n)=O(n)$ as
$n\rightarrow+\infty$. Finally, as the norms $\varphi_n$ satisfy condition (3) of theorem \ref{Thm:energieequilibree}, 
the convergence of $(\widehat{\mu}(V_n,\varphi_n,\psi_n)/n)_{n\geqslant 1}$ as a consequence of this theorem.
\end{proof}
\begin{rema}
This result invites comparison with a result of Witt Nystr\"om's \cite[th\'eor\`eme 1.4]{Nystrom11}. Both methods use 
the monomial basis to construct super or subadditive functions on the Okounkov semi-group. However, the method in 
\cite{Nystrom11} is based on a comparison between the $L^2$ metric and the $L^\infty$ metric, whereas we use the 
Harder-Narasimhan formalism. This new approach is highly flexible and enables us to prove our result in the very 
general setting of a submultiplicatively normed linear system satisfying moderate conditions, in both the complex and
non-archimedean cases.
\end{rema}

\section{Asymptotic distributions of logarithmic sections}\label{Sec:demduthm}

In this section we prove our main theorem.

\subsection{A convergence criterium.} In this section we prove a convergence criterium. 
For any real number $x$ the expression  $x_+$ denotes $\max(x,0)$.

\begin{prop}\label{Pro:critereconve}
Let $(Z_n)_{n\geqslant 1}$ be a sequence of uniformly bounded random variables. Assume that for any $t\in\mathbb R$ the
sequence $(\mathbb E[\max(Z_n,t)])_{n\geqslant 1}$ converges in $\mathbb R$. The sequence of random variables 
$(Z_n)_{n\geqslant1}$ then converges in law.
\end{prop}
\begin{proof} Note that the condition of the proposition implies that, for any $t\in\mathbb R$, the sequence $(\mathbb E[(Z_n-t)_+])_{n\geqslant 1}$ converges in $\mathbb R$. In fact, one has \[\max(Z_n,t)=(Z_n-t)_++t.\]
Let $h$ be a compactly supported smooth function. We have that
\[\begin{split}\mathbb E[h(Z_n)]&=-\int_{\mathbb R}h(t)\,\mathrm{d}\mathbb P(Z_n\geqslant t)=\int_{\mathbb R}\mathbb P(Z_n\geqslant t)h'(t)\,\mathrm{d}t,
\end{split}\]
where the second equality comes from integration by parts. For any $a\in\mathbb R$ we have that 
\[\int_{a}^{+\infty}\mathbb P(Z_n\geqslant t)\,\mathrm{d}t=\mathbb E[(Z_n-a)_+]\]
and it follows that
\[\mathbb E[h(Z_n)]=-\int_{\mathbb R}h'(t)\,\mathrm{d}\mathbb E[(Z_n-t)_+]=
\int_{\mathbb R}h''(t)\mathbb E[(Z_n-t)_+]\,\mathrm{d}t.\]
 variables $Z_n$ are uniformly bounded the dominated convergence theorem implies that the sequence 
$(\mathbb E[h(Z_n)])_{n\geqslant 1}$ converges in $\mathbb R$. Let $I(h)$ be its limit. The functional $I(.)$ is 
continuous with respect to the supremum norm on the space of compactly supported smooth functions, so it can be 
extended by continuity to a positive linear form on the space of continuous compactly supported functions, and hence
defines a Radon measure on $\mathbb R$. As the random variable  $Z_n$ are uniformly bounded it follows that 
$I(.)$ is a probability measure on $\mathbb R$. The result follows.
 
\end{proof}

\subsection{Asymptotic distribution of eigenvalues} In what follows we fix a valued field $k$ which is either 
$\mathbb C$ with the usual absolute value or a complete non-archimedean field. Let $X$ be an integral projective 
scheme of dimension $d\geqslant 1$ defined over $k$ with a regular rational point. In this subsection we prove the 
following theorem (cf. \S\ref{Subsec:polygonenorme} and \S\ref{Sec:degrehn} for notations).
\begin{theo}\label{Thm:thmprincipal}
Let $L$ be an invertible $\mathcal O_X$-module and $V_\sbullet$ a graded linear subsystem of $L$ whose Okounkov 
semi-group satisfies conditions (a)--(c) of \S\ref{Subsec:thmlimite}. Let $\varphi$ and $\psi$ be two continuous 
metrics on $L$, and for any integer $n\geqslant 0$ let $\varphi_n$ and $\psi_n$ be the supremum norms on $V_n$ induced 
by the tensor product metrics $\varphi^{\otimes n}$ and $\psi^{\otimes n}$ respectively. Then we have that
\begin{enumerate}[(1)]
\item The sequence of random variables $(\frac 1nZ_{(V_n,\varphi_n,\psi_n)})_{n\geqslant 1}$ converges in law to a probability
measure on $\mathbb R$.
\item the sequence of polygons $(P_{(V_n,\varphi_n,\psi_n)})_{n\geqslant 1}$ converges uniformly to a concave function on 
$[0,1]$;
\end{enumerate}
\end{theo}
\begin{proof}
For any $n\in\mathbb N$, $n\geqslant 1$ we let $Z_n$ denote the random variable $\frac 1n Z_{(V_n,\varphi_n,\psi_n)}$. We
have that 
\[|Z_n|\leqslant \sup_{x\in X^{\mathrm{an}}}d(\|.\|_\varphi(x),\|.\|_{\psi}(x))\]
for any $n\geqslant 1$. The sequence of random variables $(Z_n)_{n\geqslant 1}$ is therefore uniformly bounded. By 
\cite[proposition 1.2.9]{Chen10b}, the second statement follows from the first.

We now prove the first statement by using the convergence criterion given in \ref{Pro:critereconve} and the
limit result proved in \ref{Cor:thmlimite}. For any real parameter $a$ let $\varphi(a)$ be the continuous metric on
$L$ such that
\[\forall\,x\in X^{\mathrm{an}},\quad \|.\|_{\varphi(a)}(x)=\mathrm{e}^a\|.\|_{\varphi}(x).\]
Let $\psi\vee\varphi(a)$ be the metric on $L$ such that
\[\forall\,x\in X^{\mathrm{an}},\quad \|.\|_{\psi\vee\varphi(a)}(x)=
\max\big(\|.\|_{\psi}(x),\|.\|_{\varphi(a)}(x)\big).\]
The supremum norm on $V_n$ associated to the metric $(\psi\vee\varphi(a))^{\otimes n}$ is 
$\psi_n\vee\varphi_n(an)$ (with the notations as in \S\ref{SubSec:Troncature} or 
\S\ref{Subsec:troncaturepadique}). Corollary \ref{Cor:thmlimite} applied to $\varphi$ and $\psi\vee\varphi(a)$ proves that the sequence $(\widehat{\mu}(V_n,\varphi_n,\psi_n\vee\varphi_n(na))/n)_{n\geqslant 1}$ converges in $\mathbb R$. 
Moreover, by propositions \ref{Pro:troncatureC} and \ref{Pro:troncaturepadique} we have that 
\[\Big|\frac 1n\widehat{\mu}(V_n,\varphi_n,\psi_n\vee\varphi_n(na))-\mathbb E[\max(Z_n,a)]\Big|\leqslant \frac 1nA(r_n),\]
where $r_n=\rang(V_n)$ and 
\[\forall\,r\in\mathbb N,\; r\geqslant 1,\quad
A(r):=2\ln(r)+\frac 12\ln(2).\]
As $r_n=O(n^d)$ when $n\rightarrow+\infty$, we have that $\lim_{n\rightarrow+\infty}A(r_n)/n=0$. It follows that the 
sequence $(\mathbb E[\max(Z_n,a)])_{n\geqslant 1}$ converges. By Proposition \ref{Pro:critereconve}, the result follows.
\end{proof}

\begin{rema}\label{Rem:modification de metrique}
The above result still holds whenever we replace $\varphi_n$ and $\psi_n$ by norms
$\varphi_n'$ and $\psi_n'$ such that
\[\max(d(\varphi_n,\varphi_n'),d(\psi_n,\psi_n'))=o(n),\quad n\rightarrow+\infty,\]
and the limit laws are the same.
If we let $Z_n'$ be the random variable $\frac 1n Z_{( V_n,\varphi_n',\psi_n')}$ then for any $a\in\mathbb R$ we have that
\[\begin{split}&\quad\;\big|\mathbb E[\max(Z_n,a)]-\mathbb E[\max(Z_n',a)]\big|\\&
\leqslant
\frac 1n\big|\widehat{\mu}(V_n,\varphi_n,\psi_n\vee\varphi_n(an))-\widehat{\mu}(V_n,\varphi_n',\psi_n'\vee\varphi_n(an))\big|+\frac 1nA(r_n)\\
&\leqslant \frac 1n( d(\varphi_n,\varphi_n')+d(\psi_n,\psi_n')+A(r_n)).
\end{split}\]
Letting $n\rightarrow+\infty$ we get that
\[\lim_{n\rightarrow+\infty}\mathbb E[\max(Z_n',a)]=\lim_{n\rightarrow+\infty}\mathbb E[\max(Z_n,a)].\]
In particular, this enables us to apply the theorem to  $L^2$ norms when $k$ is the field of complex numbers - see for example Lemma 3.2 of \cite{Ber_Bou08}.
\end{rema}
\subsection*{Remerciements} We are grateful to S\'ebastien Boucksom for calling our attention to 
Berndtsson's article \cite{Ber09}. The second named author was supported during the writing of this work by the ANR CLASS.

\backmatter
\bibliography{chen}
\bibliographystyle{smfplain}

\end{document}